\documentclass[11pt,oneside,reqno]{amsart}
\usepackage{geometry}
\usepackage[draft, textwidth=1.2in]{todonotes}
\usepackage{color}\usepackage[centertags]{amsmath}
\usepackage{amsfonts}
\usepackage{amssymb}
\usepackage{amsthm,color}
\usepackage{newlfont}
\usepackage{bbm}
\usepackage{mathtools}
\usepackage{hyperref}
\hypersetup{
  colorlinks   = true, %Colours links instead of ugly boxes
  urlcolor     = blue, %Colour for external hyperlinks
  linkcolor    = red, %Colour of internal links
  citecolor   = red %Colour of citations
}
\usepackage{cleveref}
\theoremstyle{definition}
\usepackage{graphicx}
\setkeys{Gin}{width=\linewidth,totalheight=\textheight,keepaspectratio}
\graphicspath{{./images/}}

\usepackage{multicol}
\usepackage{booktabs}
\usepackage{mathrsfs}\usepackage{color}

\newcommand{\rbb}{\mathbb{R}}

\newcommand{\C}{\mathcal{C}}

\renewcommand{\L}{\mathcal{L}}

\newcommand{\W}{\mathcal{W}}

\newcommand{\B}{\mathcal{B}}
\newcommand{\la}{\langle}
\newcommand{\ra}{\rangle}

\newcommand{\xt}{\tilde{x}}
\newcommand{\vt}{\tilde{v}}

\newcommand{\vrt}{\tilde{\varrho}}
\newcommand{\vr}{\varrho}

\newcommand{\Xt}{\tilde{X}}
\newcommand{\Yt}{\tilde{Y}}
\newcommand{\Zt}{\tilde{Z}}

\newcommand{\Xcal}{\mathcal{X}}

\newcommand{\grad}{\nabla}

\newcommand{\nt}{\notag}

\newcommand{\mi}{\wedge}

\renewcommand{\d}{\text{d}}

\newcommand{\f}{\varphi}

\newcommand{\E}{\mathbb{E}}

\renewcommand{\P}{\mathbb{P}}

\newcommand{\Pcal}{\mathcal{P}}

\newcommand{\close}{\!\!\!}

\newcommand{\ddt}{\tfrac{\d}{\d t}}
\newcommand{\dds}{\tfrac{\d}{\d s}}

\newcommand{\R}{ \mathcal{R}}

\theoremstyle{plain}
\newtheorem{theorem}{Theorem}[section]
\newtheorem{corollary}[theorem]{Corollary}

\newtheorem{lemma}[theorem]{Lemma}
\newtheorem{assumption}[theorem]{Assumption}
\newtheorem{proposition}[theorem]{Proposition}
\newtheorem{choice}[theorem]{Choice}

\theoremstyle{definition}
\newtheorem{definition}[theorem]{Definition}

\newtheorem{remark}[theorem]{Remark}

\numberwithin{equation}{section}

\crefname{algorithm}{Algorithm}{Algorithms}
\crefname{assumption}{Assumption}{Assumptions}
\crefname{lemma}{Lemma}{Lemmas}
\crefname{theorem}{Theorem}{Theorems}
\crefname{remark}{Remark}{Remarks}
\crefname{corollary}{Corollary}{Corollaries}
\crefname{figure}{Fig.}{Figures}
\crefname{section}{Section}{Sections}
\crefname{proposition}{Proposition}{Propositions}
\crefname{definition}{Definition}{Definitions}

\author{Hung D.~Nguyen$^1$ and Anand U. Oza$^2$}

\address{$^1$ Department of Mathematics, University of Tennessee, Knoxville, Tennessee, USA}
\address{$^2$ Department of Mathematical Sciences \& Center for Applied Mathematics and Statistics, New Jersey Institute of Technology, Newark, New Jersey, USA}

\begin{document}

\title{Exponential mixing in a hydrodynamic pilot--wave theory with singular potentials}

\begin{abstract}

We conduct an analysis of a stochastic hydrodynamic pilot-wave theory, which is a Langevin equation with a memory kernel that describes the dynamics of a walking droplet (or ``walker") subjected to a repulsive singular potential and random perturbations through additive Gaussian noise. Under suitable assumptions on the singularities, we show that the walker dynamics is exponentially attracted toward the unique invariant probability measure. The proof relies on a combination of the Lyapunov technique and an asymptotic coupling specifically tailored to our setting. We also present examples of invariant measures, as obtained from numerical simulations of the walker in two-dimensional Coulomb potentials. Our results extend previous work on the ergodicity of stochastic pilot-wave dynamics established for %the instance of 
smooth confining potentials.
\end{abstract}

\maketitle

\section{Introduction}\label{Sec:Intro}

We are interested in the following stochastic integro-differential equation in $\rbb^d$:
\begin{align} \label{eqn:droplet:original}
\d x(t)&=v\,\mathrm{d}t,\\
m\,\d v(t)&=- v\,\mathrm{d}t-\grad U(x(t))\,\mathrm{d}t- \grad G(x(t))\,\d t+\int_{-\infty}^t\close H(x(t)-x(s))K(t-s)\,\mathrm{d}s\,\mathrm{d}t+\sigma\mathrm{d}W(t),\nt
\end{align}
which has been used to model the dynamics of a millimetric droplet (or ``walker") moving along the surface of a vibrating fluid bath of the same fluid~\cite{BushOzaROPP,Couder2005a,Oza2013}. In \eqref{eqn:droplet:original}, the pair $(x(t),v(t))\in \rbb^d\setminus\{0\}\times\rbb^d$ represents the displacement and velocity of the walker whose mass is denoted by the constant $m>0$, $U:\rbb^d\to [0,\infty)$ is a smooth potential satisfying certain polynomial growth, $G:\rbb^d\setminus\{0\}\to \rbb$ is a singular potential, $H:\rbb^d\to \rbb^d$ is the pilot-wave force with $K:[0,\infty)\to[0,\infty)$ the corresponding memory kernel, and $W(t)$ is a standard $d$-dimensional Brownian motion with noise strength constant $\sigma>0$. 

\subsection{Physical background} 

Over the past two decades, a number of experiments have demonstrated that a millimetric oil droplet bouncing on the surface of a vertically vibrating fluid bath exhibits certain phenomena previously thought to be exclusive to the microscopic quantum realm~\cite{Bush2024,BushOzaROPP}. These include quantized orbital radius~\cite{Blitstein2024,Fort2010,Harris2014a} and angular momentum \cite{Durey2017,Perrard2014}, wavelike statistics in cavities \cite{Gilet2016,Harris2013a,Saenz2018}, tunneling~\cite{Eddi2009b,Hubert2017,Nachbin2017,Tadrist2019}, Friedel oscillations~\cite{Saenz2019a} and Anderson localization~\cite{Abraham2024}, among others~\cite{Eddi2012,Frumkin2023a,Frumkin2023b,Frumkin2022,
Nachbin2022,Papatryfonos2022,Valani2018_2}. The droplet bounces on the surface of the bath and thus generates a guiding or `pilot' wave field. This wave field imparts a propulsive horizontal force to the droplet, which causes it to move (or ``walk") along the surface of the bath.
  
While the walker's dynamics may be chaotic, its statistics often displays wave-like patterns reminiscent of those exhibited by quantum particles. Specifically, experiments and numerical simulations have suggested that the walker's long-time statistical behavior exhibits a wave-like signature in a rotating frame~\cite{Harris2014a,Oza2014b}, harmonic potential~\cite{Durey2017,Kurianski2017,Perrard2014} and corral geometry~\cite{Cristea2018,Gilet2016,Harris2013a,Rahman2018,Saenz2018}. There has thus been an interest in understanding the long-time behavior of this hydrodynamic pilot-wave system. To that end, prior studies established a link between the walker's position probability density and the time-averaged pilot-wave field~\cite{Durey2018,Tambasco2018a}. 

 The discovery of this hydrodynamic pilot-wave system motivated the development of a number of theoretical models that describe the walker dynamics~\cite{BushOzaROPP,Turton2018}. The so-called `stroboscopic model' is the focus of this paper, and it was developed in Ref.~\cite{Oza2013} under the assumptions of no external force ($U=G\equiv 0$) and no stochastic forcing ($\sigma=0$). The model neglects consideration of the walker's vertical dynamics, and instead describes its horizontal position $x(t)\in\mathbb{R}^2$ and velocity $v(t)\in\mathbb{R}^2$ along the bath surface. Moreover, the timescale over which $x(t)$ evolves is assumed to be long relative to the vertical bouncing period. The walker is thus modeled as a moving source that continuously generates standing waves; these waves have a prescribed spatial profile $h(x)$, and decay in time due to the influence of viscosity according to the function $K(t)$~\cite{Molacek2013b}. The walker thus experiences a wave force proportional $\nabla h(x)K(t)=H(x)K(t)$, integrated over its entire past, resulting in the integral shown in Eq.~\eqref{eqn:droplet:original}. A key feature of the system is its ``path memory"~\cite{Fort2010}, as evidenced by the integral over the walker's past, with the near past having a larger influence than the far owing to the decay of $K$. The walker also experiences a drag force proportional to its velocity $v$.  
 
The stroboscopic model~\eqref{eqn:droplet:original} without stochastic forcing ($\sigma=0$) and without a singular potential ($G\equiv 0$) has been studied numerically for linear~\cite{Valani2022}, quadratic~\cite{Budanur2019,Labousse2016a,Labousse2014a,Perrard2018}, quartic~\cite{Montes2021} or Bessel~\cite{Tambasco2018a} potentials $U(x)$. However, the case of singular potentials has received comparatively little attention: to the authors' knowledge, the only such example is Ref.~\cite{Tambasco2016}, who used numerical simulations to demonstrate that, for a two-dimensional attractive Coulomb potential, $G(x) = \log(|x|)$, the dynamics undergoes a period-doubling transition to chaos as the pilot-wave force becomes progressively stronger. The Coulomb potential is of particular interest to physicists, as the description of the hydrogen atom as given by Schr\"{o}dinger's equation is a canonical problem in quantum mechanics textbooks~\cite[Chapter VII]{CohenQM}.

 We note that in the absence of the memory term, i.e., setting $K\equiv 0$, equation \eqref{eqn:droplet:original} is reduced to the classical Langevin dynamics
\begin{align}
\d x(t)&=v\,\mathrm{d}t,\nonumber \\
m\,\d v(t)&=- v\,\mathrm{d}t-\grad U(x(t))\,\mathrm{d}t- \grad G(x(t))\,\d t+\sigma\mathrm{d}W(t), \label{eqn:Langevin:original}
\end{align}
whose large-time behavior is well understood \cite{conrad2010construction, cooke2017geometric,grothaus2015hypocoercivity,
herzog2019ergodicity,lu2019geometric}. More specifically, under a broad class of polynomial potentials $U$ and singular potentials $G$, the system \eqref{eqn:Langevin:original} is exponentially attractive toward a unique invariant probability measure in  the phase space $\R^d$ where
\begin{align} \label{form:R^d}
\R^d = \begin{cases} (0,\infty)\times \rbb ,& d=1,\\
\rbb^d\setminus\{0\}\times\rbb^d,& d\ge 2.
\end{cases}
\end{align}
Here, the difference of the phase space $\R^d$ between dimension $d=1$ and $d\ge 2$ is due to the fact that $\rbb\setminus\{0\}$ is not a connected domain. Notably, the ergodicity results established for \eqref{eqn:Langevin:original} in~\cite{conrad2010construction, cooke2017geometric,grothaus2015hypocoercivity,
herzog2019ergodicity,lu2019geometric} cover a wide range of singularities, including the Lennard-Jones function
\begin{align} \label{form:Lennard-Jones}
G(x)=\frac{c_1}{|x|^{12}}-\frac{c_2}{|x|^6},
\end{align}
and the repulsive Coulomb function
\begin{align} \label{form:Coulomb}
G(x) =\begin{cases}
-c\log|x|,& d=2,\\
\frac{c}{|x|^{d-2}},& d\ge 3.
\end{cases}
\end{align}
Turning back to \eqref{eqn:droplet:original}, while there is rich literature describing numerical simulations of~\eqref{eqn:droplet:original}, as mentioned above, comparatively less work has been done on proving analytical results about its statistical behavior. 
To the best of the authors' knowledge, results in this direction were first established in \cite{nguyen2024invariant}. In particular, for a variety of smooth potentials ($G\equiv 0$), it can be shown that \eqref{eqn:droplet:original} admits a unique invariant measure \cite{nguyen2024invariant}. However, the argument therein does not cover the setting of singularities, nor does it imply an explicit convergence rate. The main goal of our present paper is to investigate the ergodicity of \eqref{eqn:droplet:original} under the impact of singular potentials and, more importantly, to prove the existence of an exponential mixing rate. In what follows, we will provide
an overview of the main mathematical results.

\subsection{Overview of the main results.} Due to the presence of the memory, it is expected that \eqref{eqn:droplet:original} is not Markovian. In the previous work of \cite{bakhtin2005stationary, ito1964stationary, nguyen2024invariant} without the singularities, this issue %can be 
was tackled by considering the phase space of trajectories in the negative time interval $C((-\infty,0];\rbb^d\times\rbb^d)$. Roughly speaking, given an initial path $(x_0,v_0)\in C((-\infty,0];\rbb^d\times\rbb^d)$ and a finite time window $[0,t]$, we construct a solution evolving according to \eqref{eqn:droplet:original}, resulting in a trajectory belonging to $C((-\infty,t];\rbb^d\times\rbb^d)$. Then, the whole path is ``shifted" back to time $0$ so as to remain in $C((-\infty,0];\rbb^d\times\rbb^d)$. While this approach is useful for establishing strong solutions to and statistical properties of~\eqref{eqn:droplet:original} for $G\equiv 0$, it is not suitable for proving ergodicity in the presence of singularities, nor for establishing a convergence rate to equilibrium. In a related work \cite{duong2023asymptotic} concerning the generalized Langevin equation, for the case of the memory kernel being a finite sum of exponentials, one can conveniently map the system therein to a finite-dimensional system of stochastic differential equations (SDEs). The approach however relies on the so-called fluctuation-dissipation relationship, which is not available in the setting of \eqref{eqn:droplet:original}.

To circumvent these difficulties, in the present article, we draw upon the framework of \cite{bonaccorsi2012asymptotic,
glatt2022short,nguyen2023ergodicity} by augmenting the original process $(x(t),v(t))$ with a ``history" variable $\eta(t)$. More precisely, we introduce $\eta(t)=\eta(t;s)$ given by
\begin{align*}
\eta(t;s) = x(t-s),\quad s, t\ge 0. 
\end{align*} 
We note that $\eta(t;s)$ satisfies the transport equation
\begin{align} \label{eqn:eta:original}
\partial_t \eta(t;s) = -\partial_s \eta(t;s),\quad \eta(t;0)=x(t).
\end{align}
To see the role of $\eta(t;s)$ in \eqref{eqn:droplet:original}, by making a simple change of variable $r =t-s$, the integral on the right-hand side of \eqref{eqn:droplet:original} is equivalent to
\begin{align*}
\int_{-\infty}^t\close H(x(t)-x(s))K(t-s)\,\mathrm{d}s& = \int_{0}^\infty\close H(x(t)-x(t-r))K(r)\,\mathrm{d}r \\
&= \int_{0}^\infty\close H(x(t)-\eta(t;r))K(r)\,\mathrm{d}r. 
\end{align*}
So, we may recast the $v$-equation of \eqref{eqn:droplet:original} as
\begin{align}
m\,\d v(t)&=- v\,\mathrm{d}t-\grad U(x(t))\,\mathrm{d}t- \grad G(x(t))\,\d t+\sigma\mathrm{d}W(t), \label{eqn:droplet:original:eta}\\
&\qquad +\int_{0}^\infty\close H(x(t)-\eta(t;r))K(r)\,\mathrm{d}r\,\mathrm{d}t.\nt
\end{align}
This results in a jointly Markov process $(x(t),v(t),\eta(t))$ evolving in an extended phase space, which will be described later in Section \ref{sec:main-result:functional-setting}. It is worth pointing out that the extended phase space approach in this work as well as in \cite{bonaccorsi2012asymptotic, nguyen2023ergodicity,nguyen2024invariant} is valid thanks to the critical assumption that the memory kernel is well-defined on $[0,\infty)$ and decays exponentially fast, cf., Assumption \ref{cond:K}. While the latter assumption is restrictive from a mathematical standpoint, it is physically well-motivated, as fluid viscosity is known to lead to exponential decay in time of the surface waves~\cite{Eddi2011a,Molacek2013b}. Moreover, the ergodicity results established below in Theorem \ref{thm:heuristic} can accommodate a rather wide class of nonlinear functions $U,G$ and $H$.

We note that while augmenting by $\eta$ is useful for analyzing ergodicity, we are mainly interested in the long-time statistical behavior of the pair $(x(t),v(t))$, i.e., the displacement and velocity of the walker. Our main result of the article concerning the convergence rate in terms of observables can be summarized as follows:

\begin{theorem} \label{thm:heuristic}
Under suitable assumptions on the functions $U$, $G$, $H$ and $K$, and suitable initial data $(x_0,v_0)$, let $(x(t),v(t))$ be the solution of \eqref{eqn:droplet:original}. Then, there exists a unique probability measure $\nu$ on $\R^d$ such that for every suitable observable $f\in C(\R^d;\rbb)$, the following holds:
\begin{align} \label{ineq:geometric-ergodicity:heuristic}
\Big|\E\big[ f(x(t),v(t))\big] -\int_{\R^d}f(x,v)\nu( \textup{d} x,\textup{d} v)\Big|\le Ce^{-ct},\quad t\ge 0,
\end{align}
for some positive constants $C$ and $c$ independent of $t$.
\end{theorem}
Here, $\E$ denotes the expectation with respect to the probability measure induced by $W(t)$. We refer the reader to Corollary \ref{cor:geometric-ergodicity} for a more precise version of Theorem \ref{thm:heuristic}. We note that Theorem \ref{thm:heuristic} extends an ergodicity result in \cite{nguyen2024invariant} established only for a polynomial potential $U$ and pilot-wave force $H$. In particular, the result presented in Theorem \ref{thm:heuristic} applies to a variety of singular potentials, e.g., the Riesz-type functions (i.e., $G\sim |x|^{-q}$ for $q > 0$ as $x\rightarrow 0$), of which the Lennard-Jones function \eqref{form:Lennard-Jones} and the Coulomb function \eqref{form:Coulomb} in dimension $d\ge 3$ are members, as well as the log function $G(x)=-\alpha\log|x| $ for some constant $\alpha>0$ sufficiently large. Also, thanks to the suitable energy estimates performed in Section \ref{sec:moment} below, \eqref{ineq:geometric-ergodicity:heuristic} is valid for a broad class of observables, including those having exponential growth. See Remark \ref{rem:G} and Remark \ref{rem:observable} for a further discussion of these points.

Historically, SDEs with memory similar to \eqref{eqn:droplet:original} were studied as early as in the seminal work \cite{ito1964stationary}. Since then, the stationary solution approach therein has been employed in \cite{bakhtin2005stationary,weinan2001gibbsian,
weinan2002gibbsian,herzog2021gibbsian,nguyen2024invariant} to investigate unique ergodicity. On the other hand, the extended phase space approach as in \eqref{eqn:droplet:original:eta} dates back to at least the work \cite{miller1974linear} and appeared later in \cite{bonaccorsi2012asymptotic, engel2000one, nguyen2023ergodicity,nguyen2024invariant}. Appending $\eta$ to $(x,v)$ allows one to rewrite \eqref{eqn:Langevin:original} as an appropriate Cauchy system while preserving the Markovian framework, which is very convenient for exploring mixing properties. However, the introduction of $\eta$ unfortunately induces an infinite dimensional system, which in comparison with \eqref{eqn:Langevin:original} requires a more careful analysis so as to establish geometric ergodicity. 

More specifically, for finite dimensional SDEs such as \eqref{eqn:Langevin:original}, mixing results are typically a consequence of suitable Lyapunov functions. In other words, the returning time to the center of the phase space is shown to be exponentially fast. Together with the strong Feller property and an irreducibility condition, these three properties are sufficient to establish a convergence rate with respect to a weighted total variation norm in terms of the Lyapunov function \cite{khasminskii2011stochastic}. We remark that while verifying the latter two properties is usually straightforward, the Lyapunov construction itself is much more nontrivial \cite{duong2023asymptotic,herzog2019ergodicity, lu2019geometric}, owing to the presence of singularities.

On the other hand, concerning \eqref{eqn:droplet:original} and \eqref{eqn:droplet:original:eta}, since this is an infinite-dimensional system with degenerate noise, it typically precludes access to the strong Feller property, let alone measuring the convergence rate with respect to total variation norm. Nevertheless, in this work, we effectively overcome the challenge by resorting to the framework of \emph{asymptotic coupling} developed in \cite{butkovsky2020generalized,hairer2011asymptotic} for stochastic PDEs. More precisely, instead of exploring total variation distance between probability measures, we opt for a more suitable Wasserstein metric defined in terms of a carefully chosen distance-like function $\vr$, cf. Section \ref{sec:main-result:geometric-ergodicity}. The convergence argument then relies on two crucial ingredients. The first is the \emph{contracting} property of $\vr$ for the Markov semigroup associated with $(x,v,\eta)$. The second is the $\vr$-\emph{small} property of bounded sets, to which the dynamics returns to infinitely often. (See Definition \ref{def:Lyapunov-contracting-dsmall} part 2. and part 3., respectively). The first can be achieved by establishing a so-called asymptotic strong Feller property, cf. Lemma \ref{lem:droplet:asymptotic-Feller}, which can be regarded as a large-time smoothing effect of the Markov semigroup. In turn, this follows from a delicate estimate on the derivative of the solutions with respect to initial data. The second can be validated by proving an irreducibility condition, cf. Proposition \ref{prop:irreducible}. In order to do so, we compare the walker dynamics with the classical Langevin equation and show that they behave similarly on any finite time window. It is important to point out that in both auxiliary results, we have to deal with corresponding control problems to ensure legitimate changes of measures via Girsanov's Theorem. We refer the reader to Proposition \ref{prop:contracting-dsmall} for the precise statements of the \emph{contracting} and $\vr$-\emph{small} properties, while their detailed proofs will be supplied at the end of Section \ref{sec:ergodicity}.

\subsection{Organization of the paper} 

The rest of the paper is organized as follows. In Section~\ref{sec:main-result}, we introduce the relevant function spaces (Section~\ref{sec:main-result:functional-setting}) and the assumptions (Section~\ref{sec:main-result:assumption}). We also state the main result of the paper, Theorem~\ref{thm:geometric-ergodicity}, which establishes the existence and uniqueness of the invariant probability measure $\nu$ for \eqref{eqn:droplet:original}, as well as the exponential mixing rate toward $\nu$. In Section~\ref{sec:moment}, we collect useful moment bounds on the solutions. Theorem~\ref{thm:geometric-ergodicity} is proved in Section~\ref{sec:ergodicity}. Numerical examples of the invariant measure for the case of a repulsive Coulomb potential in dimension $d=2$ are provided in Section~\ref{sec:numerics}.

\section{Main results} \label{sec:main-result}
\subsection{Functional setting} \label{sec:main-result:functional-setting}

We start by discussing the functional settings which will be employed for the analysis. We use $|\cdot|$ and $\la\cdot\,,\cdot\ra$, respectively to denote the Euclidean norm and inner product in $\rbb^d$. For $q\in (1,\infty)$, we introduce the memory space $\C_{q}$ given by
\begin{equation} \label{form:C_K}
\C_{q} =\Big\{\eta:[0,\infty)\to\rbb^d, \eta \text{ is measurable}, \|\eta\|_{q}\overset{\text{def}}{=}\Big|\int_0^{\infty}\close |\eta(r)|^{q}K(r)\d r\Big|^{1/q} <\infty\Big\},
\end{equation}
which is a Banach space (\cite[Theorem 3.13]{rudin1987real}) and whose dual is given by $\C_q^*= \C_p$, where $q^{-1}+p^{-1}=1$ (\cite[Theorem 6.16]{rudin1987real}). Also, for $\f\in \C^*_q=\C_p$, we denote the action of $\f$ on $\eta$ by $\la \f,\eta\ra_{q}$. Since we may identify $\f$ as an element in $\C_p$, the action $\la \f,\eta\ra_{q}$ may be understood as an integral, namely,
\begin{align}
\la \f,\eta\ra_{q}=\int_0^\infty\close \la \f(s),\eta(s)\ra K(s)\d s.\label{eqn:inner_product}
\end{align}
The product phase space for the dynamics~\eqref{eqn:droplet:original} is denoted as 
\begin{equation} \label{form:X_q}
\Xcal_{q} = \R^d\times \C_{q},
\end{equation}
where we recall $\R^d$ as in \eqref{form:R^d} is the phase space for the pair $(x,v)$.
We will also use $\|\cdot\|_{\Xcal_q}$ to denote the norm in $\Xcal_q$ given by
\begin{align}\label{form:norm:X_q}
\|(x,v,\eta)\|_{\Xcal_{q}}=|x|+|v|+\|\eta\|_{q},\quad (x,v,\eta)\in\Xcal_{q}.
\end{align}
The projections of $(x,v,\eta)$ onto marginal spaces are denoted by $\pi$, namely,
\begin{align} \label{form:pi}
\pi_x(x,v,\eta)=x,\quad \pi_v(x,v,\eta)=v,\quad \pi_\eta(x,v,\eta)=\eta.
\end{align}
For any bounded linear map $f\in \Xcal_{q}^*=L(\Xcal_{q};\rbb)$, the action of $f$ on $(x,v,\eta)$ is denoted by $\langle f,(x,v,\eta)\rangle_{\Xcal_q} $.

Following the framework of~\cite{bonaccorsi2012asymptotic,caraballo2007existence,
caraballo2008pullback,engel2000one,miller1974linear,
nguyen2024invariant}, we introduce the memory variable $\eta(t)$ given by
\begin{equation} \label{form:eta}
\eta(t;s)=x(t-s),\quad t,\,s\ge 0,
\end{equation}
which captures the past information of $x(t)$. For every $\eta_0\in \C_{q}$ and $x\in L^q([0,t];\rbb^d)$, observe that formally $\eta(t)$ obeys the following transport equation \cite{bonaccorsi2012asymptotic,pruss2013evolutionary}
\begin{align} \label{eqn:eta}
\partial_t \eta(t;\cdot)=-\partial_s \eta(t;\cdot),\quad \eta(t;0)=x(t),\quad \eta(0)=\eta_0.
\end{align}
Since the parameters $m$ and $\sigma$ in \eqref{eqn:droplet:original} do not affect
the analysis, we set $m=\sigma=1$ for the sake of simplicity. From equation~\eqref{eqn:eta} as well as \eqref{eqn:droplet:original:eta}, we may recast~\eqref{eqn:droplet:original} as follows: 
\begin{align}\label{eqn:droplet}
\d x(t)&=v(t)\d t, \nt \\
\d v(t)& = -v(t)\d t-\grad U(x(t))\d t-\grad G(x(t))\d t-\int_0^\infty\close H(x(t)-\eta(t;s))K(s)\d s\d t+\d W(t),\\
\d \eta(t)&=-\partial_s \eta(t)\d t,\quad \eta(t;0)=x(t).\nt 
\end{align}

\subsection{Main assumptions and well-posedness} \label{sec:main-result:assumption}
In this subsection, we state the main assumptions that will be employed throughout the analysis. We start with the memory kernel $K$ and impose the following condition:

\begin{assumption} \label{cond:K}
The memory kernel $K\in C^1([0,\infty);\rbb^+)$ satisfies 
\begin{align} \label{cond:K:1}
K'(t)\le -\delta K(t),\quad t\ge 0,
\end{align}
for some positive constant $\delta>0$.
\end{assumption}

\begin{remark} \label{rem:K}
1. We note that the condition~\eqref{cond:K:1} is crucial to establish useful moment bounds on the solutions, cf., Lemma \ref{lem:moment-bound}. In turn, the energy estimates will be invoked in Section \ref{sec:ergodicity} to prove the convergence rate toward equilibrium.

It is also worth mentioning that the differential inequality  \eqref{cond:K:1} implies that $K$ decays exponentially fast, i.e.,
\begin{align*}
K(t)\le K(0)e^{-\delta t},\quad t\ge 0,
\end{align*}
by virtue of Gr\"{o}nwall's inequality. However, \eqref{cond:K:1} is slightly more general than the above estimate and will be directly employed to produce suitable bounds on the variable $\eta(t)$, for instance in Eq.~\eqref{ineq:L.|eta|_(p_2)^(p_2)}. 

2. It is also important to point out that Assumption \ref{cond:K} does not cover singular memory kernels, e.g., $K(t) = \frac{e^{-\delta t}}{t^{\alpha}}$, $\alpha\in(0,1)$. Recently, an alternative method was developed in the work of \cite{liu2024dynamics,xu2022asymptotic,xu2024dynamics} to handle this types of irregular kernels. However, the approach therein seems to be only applicable to linear convolution whereas the pilot-wave function $H$ in equation \eqref{eqn:droplet} is a nonlinear function
, cf. Assumption \ref{cond:H} below. It is therefore not immediately clear how to adapt the technique of \cite{liu2024dynamics,xu2022asymptotic,xu2024dynamics} to our setting. We opt for leaving the issue of a singular memory kernel for future work.  

\end{remark}

With regard to the pilot-wave function $H$, we assume the following standard condition: %\cite{oza2014pilotA,oza2013trajectory,oza2014pilotB}

\begin{assumption}\label{cond:H} $H\in C^1(\rbb^d;\rbb^d)$ satisfies
\begin{align} \label{cond:H:1}
\max\{|\grad H(x)|, |H(x)|\}\le a_H(|x|^{p_1}+1),\quad x\in\rbb^d,
\end{align}
for some constants $a_H>0$ and  $p_1\ge 0$. 
\end{assumption}

We note that the condition we impose on $H$ is quite general. In particular, it includes the function $H(x) = \text{J}_1(|x|)x/|x|$, where $\text{J}_1$ is the Bessel function of the first kind of order one, which has been employed in models and simulations of the pilot-wave dynamics of walkers~\cite{Eddi2011a,Molacek2013b,Oza2013,Turton2018}.

Next, we state the main assumption on the potential $U$ \cite{mattingly2002ergodicity,pavliotis2014stochastic}:

\begin{assumption} \label{cond:U} 1. The potential $U\in C^\infty(\rbb^d;[1,\infty))$ satisfies, for all $x\in\rbb^d$,
\begin{align}
\frac{1}{a_0}|x|^{q_0}-1\le  |U(x)| & \le  a_0(|x|^{q_0}+1), \label{cond:U(x)<x^q}\\
|\grad U(x)| & \le  a_0(|x|^{q_0-1}+1), \label{cond:U(x)<x^q-1}\\
\text{and}\quad|\grad^2 U(x)| & \le  a_0(|x|^{q_0-2}+1), \label{cond:U(x)<x^q-2}
%\sup_{t,s\in[0,1]}U(tx+sy)&\le a_0( U(x)+U(y))
\end{align}
for some positive constants $a_0$ and $q_0$, where $\nabla^2U$ denotes the Hessian of $U(x)$. Furthermore, there exists a positive constant $\varepsilon_1$ such that 
\begin{align}\label{cond:q_0}
q_0\ge 2\max\{1,p_1+\varepsilon_1\}>\max\{2p_1,p_1+1\},
\end{align}
where $p_1$ is as in Assumption~\ref{cond:H}. 
 
2. There exist positive constants $a_1,a_2$ such that
\begin{align} \label{cond:U:xU(x)>U}
\la x,\grad U(x)\ra\ge a_1 |x|^{q_0}-a_2,\quad x\in\rbb^d.
\end{align}
\end{assumption}

\begin{remark} \label{rem:U} While most of the conditions in Assumption \ref{cond:U} are standard, the condition \eqref{cond:q_0} states that the potential $U$ must dominate the pilot-wave force induced by $H$. For example, when $H=\text{J}_1$ which is a bounded function ($p_1=0$), then $U$ is assumed to have at least quadratic growth. In case $H$ is Lipschitz ($p_1=1$), $U$ is required to grow faster than a quadratic function.
\end{remark}

Concerning the singular potential $G$, we will make the following assumption~\cite{duong2023asymptotic,duong2024trend}.

\begin{assumption} \label{cond:G} 1. $G\in C^\infty(\rbb^d\setminus\{0\};\rbb)$ satisfies $G(x)\to \infty$ as $|x|\to 0$. Furthermore, there exist constants $\beta_1\ge 1$ and $a_3>0$ such that for all $x\in\rbb^d\setminus\{0\}$
\begin{align} 
|G(x)|&\le a_3\Big(1+|x|+\frac{1}{|x|^{\beta_1}}\Big) ,\label{cond:G:G<1/|x|^beta}\\
|\grad G(x)| &\le a_3\Big(1+\frac{1}{|x|^{\beta_1}}\Big),\label{cond:G:grad.G(x)<1/|x|^beta}\\
\text{and} \quad |\grad^2 G(x)|& \le a_3\Big(1+\frac{1}{|x|^{\beta_1+1}}\Big).\label{cond:G:grad^2.G(x)<1/|x|^beta}
\end{align}

2. There exist nonnegative constants $\beta_2\in[0,\beta_1)$, $a_4$ and $a_5$ such that 
\begin{equation} \label{cond:G:|grad.G(x)+q/|x|^beta_1|<1/|x|^beta_2}
\Big|\grad G(x) +a_4\frac{x}{|x|^{\beta_1+1}}\Big| \le \frac{a_5}{|x|^{\beta_2}}+a_5, \quad x\in\rbb^d\setminus\{0\}.
\end{equation}

3. There exists a positive constant $a_6$ such that
\begin{align} \label{cond:G:e^G>grad^2.G}
1+\frac{e^{G(x)}}{|x|^{\beta_1}}\ge a_6 |\grad^2 G(x)|^2, \quad x\in\rbb^d\setminus\{0\}.
\end{align}

\end{assumption}

\begin{remark} \label{rem:G}
Intuitively, part 1 in Assumption~\ref{cond:G} controls the behavior of the singular potential $G$ both as $x\rightarrow 0$ and $x\rightarrow\infty$, whereas part 2 implies that $\nabla G\sim-|x|^{-\beta_1}$ as $x\rightarrow 0$. Moreover, part 3 says that the singular potential must be sufficiently steep, a technical condition that allows us to ensure that $x(t)$ does not get too close to the origin.

We note that the condition \eqref{cond:G:|grad.G(x)+q/|x|^beta_1|<1/|x|^beta_2} is standard and will be employed to construct Lyapunov functions in Lemma \ref{lem:moment-bound}, part 1. On the other hand, condition \eqref{cond:G:e^G>grad^2.G} is actually crucial and will be employed to produce moment bounds on the derivatives of the solution with respect to initial data, ensuring the asymptotic strong Feller property in Lemma \ref{lem:droplet:asymptotic-Feller}. See \eqref{ineq:droplet:asymptotic-strong-Feller:|grad.G|^2} in particular. Altogether, they will be invoked to establish the main result on the convergence rate toward equilibrium. Also, while  the condition \eqref{cond:G:e^G>grad^2.G} is valid for both the Lennard-Jones function \eqref{form:Lennard-Jones} and the Coulomb function \eqref{form:Coulomb} in dimension $d\ge 3$, it is not necessarily valid for Coulomb $\log$ function in dimension $d=2$. In particular, given $G(x)=-\alpha\log|x|$ for some constant $\alpha>0$, since $\beta_1=1$, \eqref{cond:G:e^G>grad^2.G} requires that $\alpha\ge 3$.
\end{remark}

Turning to the auxiliary space $\C_q$ for the memory variable $\eta$, it is crucial to determine a suitable $\C_q$ in which $\eta(t)$ evolves. For this purpose, in view of conditions \eqref{cond:H:1} and~\eqref{cond:q_0}, we pick the parameter $p_2$ as follows: 

\begin{choice} \label{choice:p_2} The constant $p_2$ satisfies
\begin{equation} \label{cond:p_2}
2\max\{1,p_1+\varepsilon_1\}>p_2>\max\{2p_1,p_1+1\},
\end{equation}
where $p_1$ and $\varepsilon_1$ are as in~\eqref{cond:H:1} and~\eqref{cond:q_0}, respectively.
\end{choice}
Given a test function $\f\in C^1$ satisfying $\f,\f'\in \C_{p_2}^* $, we interpret equation \eqref{eqn:eta} for $\eta(t)\in\C_{p_2}$ as a weak formulation in $\C_{p_2}^*$. That is, we have formally
\begin{align*}
&\int_0^\infty \la \eta(t;s),\f(s)\ra K(s)\d s\\
&= \int_0^\infty \la \eta(0;s),\f(s)\ra K(s)\d s-\int_0^t \int_0^\infty \la \partial_s \eta(r;s),\f(s)\ra K(s)\d s\,\d r\\
&= \int_0^\infty \la \eta(0;s),\f(s)\ra K(s)\d s+K(0)\int_0^t\la \eta(r;0),\f(0)\ra\d r+\int_0^t \int_0^\infty  \la \eta(r;s),\f'(s)\ra K(s)\d s\,\d r\\
&\hspace{2cm}+\int_0^t \int_0^\infty  \la \eta(r;s),\f(s)\ra K'(s)\d s\,\d r.
\end{align*}
In the last implication above, we employed an integration by parts. Invoking $\eta(r;0)=x(r)$, we deduce that
\begin{align*}
&\int_0^\infty \la \eta(t;s),\f(t)\ra K(s)\d s\\
&=  \int_0^\infty \la \eta(0;s),\f(t)\ra K(s)\d s+K(0)\int_0^t\la x(r),\f(0)\ra\d r  +\int_0^t \int_0^\infty  \la \eta(r;s),\f'(s)\ra K(s)\d s\,\d r\\
&\hspace{2cm}+\int_0^t \int_0^\infty  \la \eta(r;s),\f(r)\ra K'(s)\d s\,\d r.
\end{align*} 

Having performed the above heuristic calculation, we are now in a position to define a solution of~\eqref{eqn:droplet}, which is analytically weak and stochastically strong. For this purpose, we will fix a stochastic basis $\mathcal{S}=\left(\Omega,\mathcal{F},\P,\{\mathcal{F}_t\}_{t\geq 0},W\right)$ satisfying the usual conditions~\cite{karatzas2012brownian}. 
\begin{definition} \label{def:solution}
Given $p_2$ as in Choice \ref{choice:p_2} and an initial condition $X_0\equiv(x_0,v_0,\eta_0)\in\Xcal_{p_2}$, a process $X(t)=(x(t),v(t),\eta(t))$ is called a solution of~\eqref{eqn:droplet} if for all $T> 0$, 
$y_1,y_2\in\rbb^d$ and test function $\f$ satisfying $\f,\f'\in \C_{p_2}^*$, the following holds $\P$-a.s. for a.e. $t\in[0,T]$
\begin{align*}
\la x(t),y_1\ra&=\la x_0,y_1\ra+\int_{0}^{t}\la v(r),y_1\ra\d r,\\
\la v(t),y_2\ra &=\la v_0,y_2\ra +\int_{0}^{t} \big\la -v(r)-\grad U(x(r))-\grad G(x(r)),y_2\big\ra\d r+\int_0^t \la y_2,\d W(r)\ra,\\
&\qquad  -\int_0^t\int_0^{\infty}\close\la  H(x(r)-\eta(r;s)) ,y_2\ra K(s)\d s\,\d r,\\
\la \eta(t),\f\ra_{p_2} 
&=  \la \eta(0),\f\ra_{p_2} +K(0)\int_0^t\la x(r),\f(0)\ra\d r +\int_0^t \int_0^\infty  \la \eta(r;s),\f'(s)\ra K(s)\d s\\
&\hspace{2cm}+\int_0^t \int_0^\infty  \la \eta(r;s),\f(s)\ra K'(s)\d s,
\end{align*}     
where $\langle\cdot,\cdot\rangle_{p_2}$ is defined in~\eqref{eqn:inner_product}. 
\end{definition}

The solutions of \eqref{eqn:droplet} in the sense of Definition \ref{def:solution} can be established following the method in \cite{nguyen2024invariant} while making use of the energy estimates in Lemma \ref{lem:moment-bound} below. Since the argument is quite standard, we refer the readers to \cite{nguyen2024invariant} for the details and will henceforth assume the well-posedness of \eqref{eqn:droplet} for every initial condition in $\Xcal_{p_2}$. As a consequence of the well-posedness of~\eqref{eqn:droplet}, we can thus introduce the Markov transition probabilities of the solution $X(t)$ by
\begin{align*}
P_t(X_0,A) :=\P\big(X(t;X_0)\in A\big),
\end{align*} 
which are well-defined for $t\ge 0$, initial condition $X_0\in \Xcal_{p_2}$ and Borel sets $A\subset \Xcal_{p_2}$.  Letting $\B_b(\Xcal_{p_2})$ denote the set of bounded Borel measurable functions $f:\Xcal_{p_2} \rightarrow \rbb$, the associated Markov semigroup $P_t:\B_b(\Xcal_{p_2})\to\B_b(\Xcal_{p_2})$ is defined as
\begin{align*}
P_t f(X_0)=\E\big[f(X(t;X_0))\big], \,\, f\in \B_b(\Xcal_{p_2}).
\end{align*}
Letting $\Pcal r(\Xcal_{p_2})$ be the space of probability measures on $\Xcal_{p_2}$, we denote by $P_t\nu$ the push-forward measure of $\nu\in\Pcal r(\Xcal_{p_2}) $ under the action of $P_t$. That is
\begin{align*}
P_t\nu(A) = \int_{\Xcal_{p_2}}\P(X(t;X_0)\in A)\nu(\d X_0),
\end{align*}
for all Borel sets $A\subset \Xcal_{p_2}$. Lastly, following the framework of \cite[Section 3.5]{da2004kolmogorov}, we denote by $\L$ the Kolmogorov operator associated with $P_t$ defined for $g\in C^2(\Xcal_{p_2})$ through the relation 
\begin{align*}
   P_t g(X_0) -g(X_0)=\E\int_0^t \L g(X(r))\d r.
\end{align*}
Particularly, in view of It\^o's formula \cite[Theorem 3.2]{da2014stochastic}, it holds that
\begin{align}\label{form:L}
\L g &:= \la v,\partial_x g\ra+\Big\la -v-\grad U(x)- \grad G(x)+\int_0^{\infty}\close H(x-\eta(s))K(s)\,\mathrm{d}s ,\partial_v g\Big\ra \nt  \\
&\qquad+\frac{1}{2} \triangle_v g + \int_0^\infty \close\la -\partial_s \eta(s),\partial_\eta g \ra K(s)\d s.
\end{align}
We note that in the last term on the above right-hand side, $\partial_\eta g$ is an element in $\C_{p_2}^*=\C_{\frac{p_2}{p_2-1}}$, the dual space of $\C_{p_2}$.

\subsection{Geometric ergodicity} \label{sec:main-result:geometric-ergodicity}

We now turn to the topic of the unique ergodicity and exponential mixing rate of~\eqref{eqn:droplet}. Recall that an element $\nu\in \Pcal r(\Xcal_{p_2} )$ is said to be {\it\textbf{invariant}} for the semigroup $P_t$ if 
\begin{align*}
P_t\nu=\nu,\quad t\ge 0.
\end{align*}
Alternatively, invariance is equivalent to the identity
\begin{align*}
\int_{\Xcal_{p_2} }\close P_t f(X)\nu(\d X)=\int_{\Xcal_{p_2} }\close f(X)\nu(\d X)
\end{align*}
being true for every $f\in \B_b(\Xcal_{p_2} )$. As mentioned in the Introduction, to address the problem of  ergodicity, we will draw upon the framework developed in \cite{hairer2006ergodicity,hairer2008spectral} and later popularized in \cite{butkovsky2020generalized,hairer2011asymptotic,
hairer2011theory,kulik2015generalized,kulik2017ergodic}, tailored to our setting. For the reader's convenience, we briefly review the theory below.

Recall that a function $\varrho:\Xcal_{p_2} \times\Xcal_{p_2} \to [0,\infty)$ is called \emph{distance-like} if it is symmetric, lower semi-continuous and $\vr(X,\Xt)=0$ if and only if $X=\Xt$; see \cite[Definition 4.3]{hairer2011asymptotic}. Let $\W_{\vr}$ be the Wasserstein distance in $\Pcal r(\Xcal_{p_2})$ associated with $\vr$ given by
\begin{align} \label{form:W_d}
\W_{\vr}(\nu_1,\nu_2)=\inf \E\, \vr(X,Y),
\end{align}
where the infimum is taken over all bivariate random variables $(X,Y)$ such that $X\sim \nu_1$ and $Y\sim \nu_2$. If $\vr$ is a metric in $\Xcal_{p_2}$, by the dual Kantorovich Theorem, $\W_{\vr}$ is equivalently defined as \cite[Theorem 5.10]{villani2009optimal}
\begin{align} \label{form:W_d:dual-Kantorovich}
\W_{\vr}(\nu_1,\nu_2)=\sup_{[f]_{\text{Lip},\vr}\leq 1}\Big|\int_{\Xcal_{p_2} }\close f(X)\nu_1(\d X)-\int_{\Xcal_{p_2} }\close f(X)\nu_2(\d X)\Big|,
\end{align}
where
\begin{align} \label{form:Lipschitz}
[f]_{\text{Lip},\vr}=\sup_{X\neq \Xt}\frac{|f(X)-f(\tilde{X})|}{\vr(X,\tilde{X})}.
\end{align}
On the other hand, if $\vr$ is not a metric but instead is a distance-like function, then the following one-sided inequality holds:
\begin{equation} \label{ineq:W_d(nu_1,nu_2):dual}
\W_{\vr}(\nu_1,\nu_2) \ge \sup_{[f]_{\text{Lip},\vr}\le 1}\Big|\int_{\Xcal_{p_2} }\close f(X)\nu_1(\d X)-\int_{\Xcal_{p_2} }\close f(X)\nu_2(\d X)\Big|.
\end{equation}
See \cite[Proposition A.3]{glatt2022mixing} for a proof  of~\eqref{ineq:W_d(nu_1,nu_2):dual}. 

The type of the Wasserstein distance $\W_{\vr}$ we are mainly interested in will be described through suitable Lyapunov structures. For this purpose, we introduce the function 
\begin{align} \label{form:Phi}
\Phi(x,v) = U(x)+G(x)+\frac{1}{2}|v|^2+\kappa\la x,v\ra-\frac{\la x,v\ra}{|x|},
\end{align}
where $\kappa>0$ satisfies the following:
\begin{choice} \label{choice:kappa}
The constant $\kappa>0$ is sufficiently small such that
\begin{align} \label{cond:kappa:1}
c_\kappa \Big(U(x)+\frac{1}{2}|v|^2\Big) \le U(x)+\frac{1}{2}|v|^2+\kappa\la x,v\ra-\frac{\la x,v\ra}{|x|} \le C_\kappa \Big(U(x)+\frac{1}{2}|v|^2\Big),
\end{align}
for some positive constants $c_\kappa,C_\kappa$ independent of $(x,v,\eta)$. Furthermore,
\begin{align} \label{cond:kappa:2}
\kappa < \min\left\{ \frac{1}{2(1+1/a_1)},\frac{a_1}{2}\right\},
\end{align}
where $a_1$ is the constant as in condition \eqref{cond:U:xU(x)>U}.
\end{choice}

\begin{remark} \label{rem:kappa}
The role of the $\kappa$-term in~\eqref{form:Phi} is to introduce dissipation due to the potential $U(x)$ for large $|x|$~\cite{mattingly2002ergodicity}, while the last term introduces dissipation due to the singular potential $G$ near the origin~\cite{herzog2019ergodicity,lu2019geometric}. The condition \eqref{cond:kappa:2} will be employed to derive an exponential bound on the solution $(x(t),v(t),\eta(t))$. See the argument for \eqref{ineq:int.e^(-alpha.t+G)/|x|} below.
\end{remark}

To establish a convergence rate for the semigroup $P_t$ associated with \eqref{eqn:droplet}, let $\Psi:\Xcal_{p_2}\to [0,\infty)$ be defined as
\begin{align} \label{form:Psi}
\Psi(x,v,\eta)=\Phi(x,v)+\|\eta\|^{p_2}_{p_2}=U(x)+G(x)+\frac{1}{2}|v|^2+\kappa\la x,v\ra-\frac{\la x,v\ra}{|x|}+\|\eta\|^{p_2}_{p_2},
\end{align}
and let $\varrho:\Xcal_{p_2}\times \Xcal_{p_2}\to[0,\infty)$ be the associated metric given by
\begin{align}\label{form:varrho}
\varrho(X,\Xt)=\inf \int_0^1 e^{\frac{1}{2}\Psi(\gamma(s))}\|\gamma'(s)\|_{\Xcal_{p_2}}\d s,
\end{align}
where the infimum is taken over all paths $\gamma\in C^1([0,1];\Xcal_{p_2})$ such that $\gamma(0)=X$ and $\gamma(1)=\Xt$. Following the framework of \cite{butkovsky2020generalized,glatt2022mixing,hairer2011asymptotic,
kulik2017ergodic,kulik2015generalized,nguyen2023small}, in our setting, for $N>0$, we consider the distance
\begin{align} \label{form:varrho_N}
\varrho_N(X,\Xt)=N\varrho (X,\Xt) \mi 1.
\end{align}
The actual convergence rate of~\eqref{eqn:droplet} toward equilibrium is measured through the distance-like function $\vrt_N$ defined as
\begin{align} \label{form:varrho_N_tilde}
\vrt_N(X,\Xt) =\sqrt{ \vr_N(X,\Xt)\big[1+e^{ \Psi(X)}+e^{ \Psi(\Xt)}   \big]}.
\end{align}

We now state the main result of the paper, which establishes the unique ergodicity and exponential mixing rate of \eqref{eqn:droplet} toward equilibrium:

\begin{theorem} \label{thm:geometric-ergodicity}
The semigroup $P_t$ admits a unique invariant probability measure $\nu$. Furthermore, for all $N$ sufficiently large 
\begin{align} \label{ineq:geometric-ergodciity}
\W_{\vrt_N}\big(P_t\nu_1,P_t\nu_2  \big) \le Ce^{-ct} \W_{\vrt_N}\big(\nu_1,\nu_2  \big),\quad t\ge 0,\,\nu_1,\nu_2\in \Pcal r(\Xcal_{p_2}),
\end{align}
for some positive constants $c$ and $C$ independent of $t,\,\nu_1$ and $\nu_2$. 
\end{theorem}

The proof of Theorem \ref{thm:geometric-ergodicity} consists of three important ingredients. The first is the Lyapunov function involving $e^{\Psi(X)}$ which is used to prove that the returning time is exponentially fast. The approach that we employ to construct the Lyapunov function is drawn from the technique in \cite{lu2019geometric} tailored to our setting. The second is the contracting property of the semigroup $P_t$ with respect to $\vr_N$. In other words, if two initial data are close enough, then so are the two solutions for all time $t\ge0$. The last property is the $\vr$-smallness of bounded sets.  
Heuristically, this means that the distance between two solutions with initial data selected from a bounded set remains uniformly bounded for all time. The Lyapunov structure will be presented in Section \ref{sec:moment} whereas the latter two properties will be supplied in Section \ref{sec:ergodicity:proof-of-theorem}. The proof of Theorem \ref{thm:geometric-ergodicity} will also be provided in Section \ref{sec:ergodicity:proof-of-theorem}.

As a consequence of Theorem \ref{thm:geometric-ergodicity} and the relation \eqref{ineq:W_d(nu_1,nu_2):dual}, we obtain the following convergence rate in terms of observables.

\begin{corollary} \label{cor:geometric-ergodicity}
For all $X_0\in\Xcal_{p_2}$ and $f\in C^1(\Xcal_{p_2};\rbb)$ satisfying $[f]_{\textup{Lip},\vrt_N}<\infty$, the following holds:
\begin{align} \label{ineq:geometric-ergodciity:P_tf}
\Big|P_t f(X_0)-\int_{\Xcal_{p_2}}\close f(X)\nu(\emph{d} X)\Big|\le Ce^{-ct},\quad t\ge 0,
\end{align}
for some positive constants $c$ and $C$ independent of $t$.

\end{corollary}

The proof of Corollary \ref{cor:geometric-ergodicity} will be given in Section \ref{sec:ergodicity:proof-of-theorem}.

\begin{remark} \label{rem:observable}
Following \cite[Proposition A.9]{glatt2022mixing}, a sufficient condition for $f\in C^1(\Xcal_{p_2};\rbb)$ satisfying $[f]_{\textup{Lip},\vrt_N}<\infty$ is that
\begin{align*}
\sup_{X\in \Xcal_{p_2}} \frac{\max\{f(X),\|Df(X)\|_{\Xcal_{p_2}^*}\}}{\sqrt{1+e^{\frac{1}{2}\Psi(X)}}}<\infty.
\end{align*}
In the above, $Df$ denotes the Fr\'echet derivative. In particular, it is not difficult to verify that the class of polynomials $f(X)$ satisfies the above condition.
\end{remark}

\section{Moment estimates} \label{sec:moment}

In this section, we establish moment bounds on the solution of~\eqref{eqn:droplet}, which will be employed to prove the main results in Section~\ref{sec:ergodicity}. Throughout the rest of the paper, $c$ and $C$ denote generic positive constants. The main parameters that they depend on will appear between parentheses, e.g., $c(T,q)$ is a function of $T$ and $q$.

\begin{lemma} \label{lem:moment-bound}
Given $(x_0,v_0,\eta_0)\in \Xcal_{p_2}=\rbb^d\times\rbb^d\times \C_{p_2}$ where $p_2$ is as in Choice \ref{choice:p_2}, let $(x(t),v(t),\eta(t))$ be the solution of \eqref{eqn:droplet}. Then, the following holds:

1. There exist positive constants $c_1$ and $C_1$ such that
\begin{align} \label{ineq:E.e^Phi(t)<C.e^(-ct).e^Phi(0)+C}
&\E \exp\big\{ \Phi(x(t),v(t))+ \|\eta(t)\|^{p_2}_{p_2}\big\}  \nt \\
& \le C_1 e^{-c_1 t}\exp\big\{  \Phi(x_0,v_0)+\|\eta_0\|^{p_2}_{p_2}\big\}+C_1,\quad (x_0,v_0,\eta_0)\in \Xcal_{p_2},\,t\ge 0,
\end{align}
where $\Phi$ is the function defined in \eqref{form:Phi}.

2. For all $\alpha$ sufficiently large, it holds that
%\begin{align} \label{ineq:E.e^(Phi+eta)<e^(alpha.t)}
%\E \exp\big\{\Phi(t)+\|\eta(t)\|^{p_2}_{p_2}\big\} \le \E \exp\big\{\alpha t+ \Phi(t)+\|\eta(t)\|^{p_2}_{p_2}\big\},\quad t\ge 0,  
%\end{align}
%and that
\begin{align} \label{ineq:int.e^(-alpha.t+G)/|x|}
\int_0^\infty\close e^{-\alpha t }\,\E \bigg[\frac{e^{G(x(t))}}{|x(t)|^{\beta_1} }\bigg] \emph{d} t  <C_\alpha \exp\big\{ \Phi(x_0,v_0)+\|\eta_0\|^{p_2}_{p_2}\big\},\quad (x_0,v_0,\eta_0)\in \Xcal_{p_2}
\end{align}
for some positive constant $C_\alpha$ independent of $(x_0,v_0,\eta_0)$. In the above, $\beta_1$ is the constant from Assumption \ref{cond:G}.
%where $\Phi$ is the function defined in \eqref{form:Phi}.

%\textup{(ii)} 
%For all $\beta$ sufficiently small, there exist positive constants $c_\beta$ and $C_\beta$ such that
%\begin{align} \label{ineq:E.e^(beta.Phi(t))<e^(-ct).e^(beta.Phi(0))}
%\E e^{\beta\Phi(x(t),v(t),\eta(t))}\le C_{\beta} e^{-c_\beta t}e^{\Phi(x,v,\eta)}+C_\beta,\quad (x,v,\eta)\in \Xcal_{p_2},\, t\ge 0.
%\end{align}
%
%\textup{(iii)} 
%For all $\beta$ and $\gamma$ sufficiently small, there exists a positive constant $C_{\beta,\gamma}$ such that
%\begin{align} \label{ineq:E.e^(beta.Phi(t))<e^(beta.e^(-gamma.t).Phi(0))}
%\E e^{\beta\Phi(x(t),v(t),\eta(t))}\le C_{\beta,\gamma} e^{\beta e^{-\gamma t}\Phi(x,v,\eta)},\quad (x,v,\eta)\in \Xcal_{p_2},\,t\ge 0.
%\end{align}

\end{lemma}

\begin{proof} 1. Recalling the operator $\L$ from \eqref{form:L}, we first apply $\L$ to $U(x)+G(x)+\frac{1}{2}|v|^2+\kappa\la x,v\ra$ and obtain the identity
\begin{align} \label{eqn:L.U+G+v^2+<x,v>}
\L \Big(U(x)+G(x)+\frac{1}{2}|v|^2+\kappa\la x,v\ra\Big) &= -(1-\kappa)|v|^2+ \int_0^\infty\close \la H(x-\eta(s)),v+\kappa x\ra K(s)\d s+\frac{d}{2} \nt \\
&\qquad -\kappa \la x,\grad U(x)+\grad G(x)\ra .
\end{align}
Since $H$ satisfies condition \eqref{cond:H:1}, we employ the Cauchy-Schwarz inequality to estimate the integral on the above right-hand side as follows:
\begin{align*}
&\left|\int_0^{\infty}\close \la H(x-\eta(s)), \kappa x+v\ra K(s)\d s\right|\\&\le c\big(|x|^{1+p_1}+|v|\,|x|^{p_1}+|x|+|v|\big)\|K\|_{L^1(\rbb^+)}+c\big(|x|+|v|\big)\int_0^{\infty}\close|\eta(s)|^{p_1}K(s)\d s.
\end{align*}
Letting $\varepsilon$ be small and chosen later, by the $\varepsilon$-Young inequality and the condition \eqref{cond:q_0} on $q_0$, we observe that
\begin{align*}
c\big(|x|^{1+p_1}+|v|\,|x|^{p_1}+|x|+|v|\big)\|K\|_{L^1(\rbb^+)} \le \frac{1}{4}\varepsilon|v|^2+\frac{1}{8}\varepsilon |x|^{q_0}+C, 
\end{align*}
for some positive constant $C=C(\varepsilon)$ that might grow arbitrarily large as $\varepsilon\rightarrow 0$. Likewise, 
\begin{align*}
c(|x|+|v|)\int_0^{\infty}\close|\eta(s)|^{p_1}K(s)\d s&\le \frac{1}{8}\varepsilon|x|^2+\frac{1}{4}\varepsilon|v|^2+\frac{1}{4}\varepsilon\int_0^{\infty}\close|\eta(s)|^{p_2}K(s)\d s+C\\
&\leq\frac{1}{8}\varepsilon|x|^{q_0}+\frac{1}{4}\varepsilon|v|^2+\frac{1}{4}\varepsilon\|\eta\|^{p_2}_{p_2}+C,
\end{align*}
where, again, $C=C(\varepsilon)$ may be arbitrarily large. Concerning the cross term $\la x,\grad U(x)+\grad G(x)\ra$ on the right-hand side of \eqref{eqn:L.U+G+v^2+<x,v>}, we invoke conditions \eqref{cond:U:xU(x)>U} and~\eqref{cond:G:grad.G(x)<1/|x|^beta} to infer the bound
\begin{align*}
-\kappa\la x, \grad U(x)+\grad G(x)\ra &\le -\kappa a_1|x|^{q_0}+ C|x|+C\frac{1}{|x|^{\beta_1-1}}+C\\
&\le -\Big(\kappa a_1-\frac{1}{4}\varepsilon\Big)|x|^{q_0}+\frac{1}{4}\varepsilon\frac{1}{|x|^{\beta_1}}+C,
\end{align*}
where the last inequality follows from the fact that $q_0\ge 2$, cf. condition \eqref{cond:q_0}, and $\beta_1\ge 1$, cf. Assumption \ref{cond:G}. Together with the identity \eqref{eqn:L.U+G+v^2+<x,v>}, we deduce the estimate
\begin{align} \label{ineq:L.U+G+v^2+<x,v>}
&\L \Big(U(x)+G(x)+\frac{1}{2}|v|^2+\kappa\la x,v\ra\Big)  \nt \\
& \le -\Big(1-\kappa-\frac{1}{2}\varepsilon\Big)|v|^2-\Big(\kappa a_1-\frac{1}{2}\varepsilon\Big)|x|^{q_0}+\frac{1}{4}\varepsilon\|\eta\|^{p_2}_{p_2} +\frac{1}{4}\varepsilon\frac{1}{|x|^{\beta_1}}+C.
\end{align}

Next, considering $\|\eta\|^{p_2}_{p_2}$, note that
\begin{align*}
\partial_\eta\big(\|\eta\|^{p_2}_{p_2}\big)=p_2|\eta|^{p_2-2}\eta,
\end{align*}
from which it follows that
\begin{align*}
\L\|\eta\|^{p_2}_{p_2} &=  -\int_0^\infty \la \partial_s\eta(s),\eta(s)\ra\cdot p_2 |\eta(s)|^{p_2-2}K(s)\d s  \\
& = -\int_0^{\infty}\close \dds\big(|\eta(s)|^{p_2}\big) K(s)\d s.
\end{align*}
Using integration by parts on the above right-hand side yields 
\begin{align}
\L\|\eta\|^{p_2}_{p_2}&=|\eta(0)|^{p_2}K(0)+\int_0^{\infty}\close |\eta(s)|^{p_2} K'(s)\d s \notag\\
&\le |\eta(0)|^{p_2}K(0)-\delta\int_0^{\infty}\close|\eta(s)|^{p_2} K(s)\d s\notag\\
&=|\eta(0)|^{p_2}K(0)-\delta\|\eta\|^{p_2}_{p_2} \nt \\
&\le \frac{1}{4}\varepsilon |\eta(0)|^{q_0}-\delta\|\eta\|^{p_2}_{p_2}+C, \label{ineq:L.|eta|_(p_2)^(p_2)}
\end{align}
where the first inequality follows from Assumption~\ref{cond:K}, namely, $K'(s)\le -\delta K(s)$.

Turning to the cross term $-\frac{\la x,v\ra}{|x|}$, applying $\L$ to $-\frac{\la x,v\ra}{|x|}$ gives
\begin{align} \label{eqn:L.<x,v>/x}
-\L \left(\frac{\la x,v\ra}{|x|}\right)  &= -\frac{|v|^2}{|x|}+\frac{|\la x,v\ra|^2}{|x|^3}+\frac{\la x,v\ra}{|x|}+\frac{\la x,\grad U(x)+\grad G(x)\ra}{|x|} \nt \\
&\qquad -\frac{1}{|x|}\int_0^\infty \la H(x-\eta(s)),x\ra K(s)\d s.
\end{align}
Using the Cauchy-Schwarz inequality, we observe that
\begin{align*}
-\frac{|v|^2}{|x|}+\frac{|\la x,v\ra|^2}{|x|^3} \le 0,
\end{align*}
and that
\begin{align*}
&\frac{\la x,v\ra}{|x|}+\frac{\la x,\grad U(x)\ra}{|x|}-\frac{1}{|x|}\int_0^\infty\close \la H(x-\eta(s)),x\ra K(s)\d s\\
&\le |v|+|\grad U(x)| + \int_0^\infty\close |H(x-\eta(s))| K(s)\d s\\
&\le \frac{1}{2}\varepsilon|v|^2+C|x|^{q_0-1}+C|x|^{p_1}+C\int_0^\infty |\eta(s)|^{p_1}K(s)\d s+C\\
&\le \frac{1}{2}\varepsilon|v|^2+\frac{1}{4}\varepsilon|x|^{q_0}+\frac{1}{4}\varepsilon\|\eta\|^{p_2}_{p_2}+C,
\end{align*}
where, in the second to last implication above, we employed \eqref{cond:U(x)<x^q-1} and \eqref{cond:H:1}. Concerning the last term involving $\grad G$ on the right-hand side of \eqref{eqn:L.<x,v>/x}, we invoke \eqref{cond:G:|grad.G(x)+q/|x|^beta_1|<1/|x|^beta_2} as follows:
\begin{align*}
\frac{\la x,\grad G(x)\ra}{|x|} &=-a_4\frac{1}{|x|^{\beta_1}}+\frac{1}{|x|}\left\la x,\grad G(x)+a_4 \frac{x}{|x|^{\beta_1+1}} \right\ra\\
&\le -a_4\frac{1}{|x|^{\beta_1}}+a_5\frac{1}{|x|^{\beta_2}}+a_5\\
&\le -\Big(a_4-\frac{1}{4}\varepsilon\Big)  \frac{1}{|x|^{\beta_1}}+C,
\end{align*}
where in the last implication, we employed the fact that $\beta_1>\beta_2$. Together with the identity \eqref{eqn:L.<x,v>/x}, we obtain
\begin{align} \label{ineq:L.<x,v>/x}
-\L \left(\frac{\la x,v\ra}{|x|}\right) \le \frac{1}{2}\varepsilon|v|^2+\frac{1}{4}\varepsilon|x|^{q_0}+\frac{1}{4}\varepsilon\|\eta(s)\|^{p_2}_{p_2} -\Big(a_4-\frac{1}{4}\varepsilon\Big)  \frac{1}{|x|^{\beta_1}}+C.
\end{align}
Now, we collect \eqref{ineq:L.U+G+v^2+<x,v>}, \eqref{ineq:L.|eta|_(p_2)^(p_2)} and \eqref{ineq:L.<x,v>/x} to arrive at the bound
\begin{align*} %\label{ineq:L.Phi}
\L\big[\Phi(x,v)+\|\eta\|^{p_2}_{p_2}\big] & \le -(1-\kappa-\varepsilon)|v|^2-\Big(\kappa a_1-\frac{3}{4}\varepsilon\Big)|x|^{q_0}-(\delta-\varepsilon)\|\eta\|^{p_2}_{p_2} -(a_4-\varepsilon)\frac{1}{|x|^{\beta_1}} \nt \\
&\qquad+\frac{1}{4}\varepsilon |\eta(0)|^{q_0} +C.
\end{align*}
As a consequence, we invoke It\^o's formula while making use of the boundary condition $\eta(t;0)=x(t)$ to see that
\begin{align} \label{ineq:d.Phi(t)}
\d \big[\Phi(t)+\|\eta(t) \|^{p_2}_{p_2}\big] & \le -(1-\kappa-\varepsilon)|v(t)|^2\d t-(\kappa a_1-\varepsilon)|x(t)|^{q_0}\d t-(\delta-\varepsilon)\|\eta(t)\|^{p_2}_{p_2}\d t \nt \\ 
&\qquad-(a_4-\varepsilon)\frac{1}{|x(t)|^{\beta_1}}\d t+C\d t+ \d M(t),
\end{align}
where the semi-Martingale process $M(t)$ is defined as
\begin{align} \label{form:M(t)}
M(t) =\int_0^ t\Big\la v(r)+\kappa x(r)+\frac{x(r)}{|x(r)|} , \d W(r) \Big\ra,
\end{align}
whose quadratic variation process is given by
\begin{align} \label{form:M(t):quadratic-variation}
\la M\ra(t)=\int_0^t \Big|v(r)+\kappa x(r)+\frac{x(r)}{|x(r)|}\Big|^2\d r.
\end{align}

Turning back to \eqref{ineq:E.e^Phi(t)<C.e^(-ct).e^Phi(0)+C}, we apply It\^o's formula to $\exp\big\{\Phi(t)+\|\eta(t)\|^{p_2}_{p_2}\big\}$ and obtain the identity
\begin{align*}
&\frac{\d \exp\big\{\Phi(t)+\|\eta(t)\|^{p_2}_{p_2}\big\} }{\exp \big\{\Phi(t)+\|\eta(t)\|^{p_2}_{p_2}\big\} }\\
& = \d \big[\Phi(t)+\|\eta(t)\|^{p_2}_{p_2}\big] +\frac{1}{2}\la \d \big[\Phi(t)+\|\eta(t)\|^{p_2}_{p_2}\big],\d \big[\Phi(t)+\|\eta(t)\|^{p_2}_{p_2}\big]\big\ra\\
&= \d \big[\Phi(t)+\|\eta(t)\|^{p_2}_{p_2}\big] +\frac{1}{2}\Big| v(t)+\kappa x(t)+\frac{x(t)}{|x(t)|} \Big|^2\d t,
\end{align*}
where the last equality follows from the expression \eqref{form:M(t):quadratic-variation}. By the $\varepsilon$-Young inequality, we have
\begin{align*}
\frac{1}{2}\Big| v+\kappa x+\frac{x}{|x|} \Big|^2 & \le \frac{1+\frac{2\kappa}{a_1} }{2}|v|^2+\frac{1+\frac{a_1}{2\kappa}}{2}\Big| \kappa x+\frac{x}{|x|} \Big|^2\\
&\le \frac{1+\frac{2\kappa}{a_1}}{2}|v|^2+\Big(1+\frac{a_1}{2\kappa}\Big) \kappa^2| x|^2 + 1+\frac{a_1}{2\kappa}\\
&\le \frac{1+\frac{2\kappa}{a_1}}{2}|v|^2+\Big(1+\frac{a_1}{2\kappa}\Big) \kappa^2| x|^{q_0} +C,
\end{align*}
for some positive constant $C$ dependent on $\kappa$ and $a_1$. We combine the above estimate with \eqref{ineq:d.Phi(t)} and the choice of $\kappa$ in \eqref{cond:kappa:2} to deduce that
\begin{align}
&\frac{\d \exp\big\{\Phi(t)+\|\eta(t)\|^{p_2}_{p_2}\big\} }{\exp \big\{\Phi(t)+\|\eta(t)\|^{p_2}_{p_2}\big\} } \nt \\
&= \d \big[\Phi(t)+\|\eta(t)\|^{p_2}_{p_2}\big] +\frac{1}{2}\Big| v(t)+\kappa x(t)+\frac{x(t)}{|x(t)|} \Big|^2\d t \nt \\
&\le -\Big(\frac{1}{2}-\kappa-\varepsilon-\frac{\kappa}{a_1}\Big)|v(t)|^2\d t-\Big(\frac{\kappa a_1}{2}-\varepsilon - \kappa^2   \Big)|x(t)|^{q_0}\d t \nt \\
&\qquad-(a_4-\varepsilon)\frac{1}{|x(t)|^{\beta_1}}\d t -(\delta-\varepsilon)\|\eta(t)\|^{p_2}_{p_2}\d t+C \d t+\d M(t) \nt \\
&\le -c\big[ \Phi(t)+\|\eta(t)\|^{p_2}_{p_2}\big]\d t +C\d t+\d M(t)   . \label{ineq:d.exp(Phi+eta)}
\end{align}
%
%
%
%
%
%letting $\beta>0$ be given and be chosen later, we apply It\^o's formula to $e^{\beta[ \Phi(t)+\|\eta(t)\|^{p_2}_{p_2}] }$ and obtain
%\begin{align*}
%\d e^{\beta [ \Phi(t)+\|\eta(t)\|^{p_2}_{p_2}]} & = \beta e^{\beta [ \Phi(t)+\|\eta(t)\|^{p_2}_{p_2}] } \d [ \Phi(t)+\|\eta(t)\|^{p_2}_{p_2}]\\
%&\qquad+\frac{1}{2}\beta^2 e^{\beta \Phi(t) } \la \d[ \Phi(t)+\|\eta(t)\|^{p_2}_{p_2}],\d[ \Phi(t)+\|\eta(t)\|^{p_2}_{p_2}]\ra.
%\end{align*}
%From \eqref{ineq:d.Phi(t)}, we see that
%\begin{align*}
%\d [ \Phi(t)+\|\eta(t)\|^{p_2}_{p_2}]\le -c[ \Phi(t)+\|\eta(t)\|^{p_2}_{p_2}]\d t+C\d t +\d M(t).
%\end{align*}
%Together with \eqref{form:M(t):quadratic-variation}, we obtain
%\begin{align*}
%\d e^{\beta [ \Phi(t)+\|\eta(t)\|^{p_2}_{p_2}]} & \le \beta e^{\beta [ \Phi(t)+\|\eta(t)\|^{p_2}_{p_2}]}\Big[ -c [ \Phi(t)+\|\eta(t)\|^{p_2}_{p_2}]\d t+C\d t+\d M(t)\\
%&\qquad\qquad + \frac{1}{2}\beta \Big| v(t)+\kappa x(t)+\frac{x(t)}{|x(t)|} \Big|^2\d t\Big].
%\end{align*}
%Picking $\beta>0$ sufficiently small, we note that
%\begin{align*}
%-c [ \Phi(x,v)+\|\eta\|^{p_2}_{p_2}] +\frac{1}{2}\beta \Big| v+\kappa x+\frac{x}{|x|} \Big|^2 \le -\tilde{ c}[ \Phi(x,v)+\|\eta\|^{p_2}_{p_2}]+\tilde{C},\quad (x,v,\eta)\in \Xcal_{p_2}.
%\end{align*}
%It follows that
%\begin{align*}
%\d e^{\beta [ \Phi(t)+\|\eta(t)\|^{p_2}_{p_2}]}  \le \beta e^{\beta [ \Phi(t)+\|\eta(t)\|^{p_2}_{p_2}] }\Big[ -c [ \Phi(t)+\|\eta(t)\|^{p_2}_{p_2}]\d t+C\d t+\d M(t) \Big].
%\end{align*}
In particular, we also have
\begin{align*}
    \L \exp\big\{\Phi(x,v)+\|\eta\|^{p_2}_{p_2}\big\} \le \exp\big\{\Phi(x,v)+\|\eta\|^{p_2}_{p_2}\big\} \Big[-c \exp\big\{\Phi(x,v)+\|\eta\|^{p_2}_{p_2}\big\}+C\Big].
\end{align*}
Since
\begin{align*}
    \E \exp\big\{\Phi(t)+\|\eta(t)\|^{p_2}_{p_2}\big\} - \E \exp\big\{\Phi(0)+\|\eta_0\|^{p_2}_{p_2}\big\} = \int_0^t \E\Big[ \L \exp\big\{\Phi(r)+\|\eta(r)\|^{p_2}_{p_2}\big\}\Big]\d r,
\end{align*}
we obtain
\begin{align*}
&\frac{\d}{\d t} \E \exp\big\{\Phi(t)+\|\eta(t)\|^{p_2}_{p_2}\big\} \\
&\le \E\Big[ \L \exp\big\{\Phi(t)+\|\eta(t)\|^{p_2}_{p_2}\big\}\Big] \\
&\le \E \Big[  \exp\big\{\Phi(t)+\|\eta(t)\|^{p_2}_{p_2}\big\}\cdot  \big(-c \big[ \Phi(t)+\|\eta(t)\|^{p_2}_{p_2}\big]+C\big)\Big].
\end{align*}
We note that the following holds:
\begin{align*}
e^{r}(-cr+C)\le -\tilde{ c} e^{r}+\tilde{C},\quad r\ge 0, 
\end{align*}
for some positive constants $\tilde{ c}$ and $\tilde{ C}$ that may depend on $c$ and $C$. In turn, this implies that
\begin{align*}
\ddt \E \exp\big\{\Phi(t)+\|\eta(t)\|^{p_2}_{p_2}\big\}  \le -c\E\exp\big\{\Phi(t)+\|\eta(t)\|^{p_2}_{p_2}\big\}  +C.
\end{align*}
By Gr\"{o}nwall's inequality, the above immediately produces the estimate \eqref{ineq:E.e^Phi(t)<C.e^(-ct).e^Phi(0)+C}, which completes the proof of part~1.

2. With regard to \eqref{ineq:int.e^(-alpha.t+G)/|x|}, letting $\alpha\ge 0$ be given and chosen later, we apply It\^o's formula to $\exp\big\{-\alpha t + \Phi(t)+\|\eta(t)\|^{p_2}_{p_2}\big\}$ and obtain the identity
\begin{align*}
&\frac{\d \exp\big\{-\alpha t +\Phi(t)+\|\eta(t)\|^{p_2}_{p_2}\big\} }{\exp \big\{-\alpha t +\Phi(t)+\|\eta(t)\|^{p_2}_{p_2}\big\} }\\
& = -\alpha \d t +\d \big[\Phi(t)+\|\eta(t)\|^{p_2}_{p_2}\big] +\frac{1}{2}\la \d \big[\Phi(t)+\|\eta(t)\|^{p_2}_{p_2}\big],\d \big[\Phi(t)+\|\eta(t)\|^{p_2}_{p_2}\big]\big\ra\\
&= -\alpha \d t +\d \big[\Phi(t)+\|\eta(t)\|^{p_2}_{p_2}\big] +\frac{1}{2}\Big| v(t)+\kappa x(t)+\frac{x(t)}{|x(t)|} \Big|^2\d t,
\end{align*}
where the last equality follows from the expression \eqref{form:M(t):quadratic-variation}. From \eqref{cond:kappa:1} and \eqref{ineq:d.exp(Phi+eta)}, we have
\begin{align*}
&\frac{\d \exp\big\{-\alpha t +\Phi(t)+\|\eta(t)\|^{p_2}_{p_2}\big\} }{\exp \big\{-\alpha t +\Phi(t)+\|\eta(t)\|^{p_2}_{p_2}\big\} } \nt \\
&= -\alpha \d t +\d \big[\Phi(t)+\|\eta(t)\|^{p_2}_{p_2}\big] +\frac{1}{2}\Big| v(t)+\kappa x(t)+\frac{x(t)}{|x(t)|} \Big|^2\d t \nt \\
&\le -\alpha \d t + C \d t-(a_4-\varepsilon)\frac{1}{|x(t)|^{\beta_1}}\d t+\d M(t). %\label{ineq:d.exp(-alpha.t+Phi+eta)}
\end{align*}
As a consequence, we may take $\alpha$ large enough to infer
\begin{align*}
\frac{\d \exp\big\{-\alpha t +\Phi(t)+\|\eta(t)\|^{p_2}_{p_2}\big\} }{\exp \big\{-\alpha t +\Phi(t)+\|\eta(t)\|^{p_2}_{p_2}\big\} }\le-(a_4-\varepsilon)\frac{1}{|x(t)|^{\beta_1}}\d t+\d M(t).
\end{align*}
In particular, we obtain the bound in expectation
\begin{align}
&\frac{\d}{\d t} \E \exp\big\{-\alpha t +\Phi(t)+\|\eta(t)\|^{p_2}_{p_2}\big\} \nt \\
&\le -(a_4-\varepsilon)\E \Big[\frac{\exp\big\{-\alpha t +\Phi(t)+\|\eta(t)\|^{p_2}_{p_2}\big\}}{|x(t)|^{\beta_1} }\Big]. \label{ineq:d/dt.E.exp(-alpha.t+Phi+eta)}
\end{align}
By integrating \eqref{ineq:d/dt.E.exp(-alpha.t+Phi+eta)} with respect to time, it follows immediately that
\begin{align*}
\int_0^t(a_4-\varepsilon) \E \Big[\frac{\exp\big\{-\alpha r +\Phi(r)+\|\eta(r)\|^{p_2}_{p_2}\big\}}{|x(r)|^{\beta_1} }\Big]\d r \le \exp\big\{\Phi(x_0,v_0)+\|\eta_0\|^{p_2}_{p_2}\big\}.
\end{align*}
Since $\Phi$ as in \eqref{form:Phi} satisfies $\Phi(x,v)\ge G(x)$, we arrive at the estimate
\begin{align*}
\int_0^t (a_4-\varepsilon) \E \left[\frac{\exp\big\{-\alpha r+G(x(r))\big\}}{|x(r)|^{\beta_1} }\right]\d r \le \exp\big\{\Phi(x_0,v_0)+\|\eta_0\|^{p_2}_{p_2}\big\}.
\end{align*}
Sending $t$ to infinity, this establishes \eqref{ineq:int.e^(-alpha.t+G)/|x|} by virtue of the Monotone Convergence Theorem. The proof is thus finished.

\end{proof}

\section{Geometric ergodicity} \label{sec:ergodicity}
In this section, we establish the main results of the paper by proving the existence and uniqueness of the invariant probability measure $\nu$, as well as establishing the exponential mixing rate toward $\nu$. As mentioned in the Introduction, we draw upon the framework of \cite{butkovsky2020generalized, hairer2011asymptotic} tailored to our setting. In Section \ref{sec:ergodicity:asymptotic-strong-feller}, we establish the asymptotic strong Feller property, which will be employed to prove the contracting property of the semigroup $P_t$. In Section \ref{sec:ergodicity:irreducibility}, we derive the irreducibility property, which is a control problem proving that the solutions can always be driven back to the center of the phase space. In turn, this will be invoked to deduce the smallness property of any bounded sets in $\Xcal_{p_2}$. In Section \ref{sec:ergodicity:proof-of-theorem}, we will establish Theorem \ref{thm:geometric-ergodicity} by using the auxiliary results collected in Section \ref{sec:moment}, Section \ref{sec:ergodicity:asymptotic-strong-feller} and Section \ref{sec:ergodicity:irreducibility}.

\subsection{Asymptotic strong Feller property} \label{sec:ergodicity:asymptotic-strong-feller}

For any unit element $\xi\in\Xcal_{p_2}$, we denote by $J_{0,t}\xi$ the derivative of the solution $(x(t),v(t),\eta(t))$ with respect to the initial condition $\xi$. Recalling the projection $\pi$ defined in \eqref{form:pi}, observe that $J_{0,t}\xi=(\pi_x J_{0,t}\xi, \pi_v J_{0,t}\xi,\pi_\eta J_{0,t}\xi)$ satisfies the following equation with random coefficients:
\begin{align}
\tfrac{\d}{\d t} \pi_x J_{0,t}\xi&= \pi_v J_{0,t}\xi, \nt\\
\tfrac{\d}{\d t} \pi_v J_{0,t}\xi&= -\pi_v J_{0,t} \xi-\grad^2 U(x(t))\pi_x J_{0,t} \xi -\grad^2 G(x(t))\pi_x J_{0,t} \xi \nt\\
&\qquad\qquad-\int_0^\infty\close \grad H(x(t)-\eta(t;s))(\pi_x J _{0,t}\xi-\pi_\eta J_{0,t} \xi(s))K(s)\d s,\label{eqn:droplet:J}\\
\ddt \pi_\eta J_{0,t} \xi&=-\partial_s\pi_\eta J_{0,t} \xi,\quad \pi_\eta J_{0,t} \xi(0)=\pi_x J_{0,t} \xi,\, t>0, \nt 
\end{align}
with the initial condition
\begin{align*}
 J_{0,0} \xi=\xi.
\end{align*}

Next, for any adapted process $\zeta\in L^2(0,t;\rbb^d)$, we introduce the Malliavin derivative $A_{0,t} \zeta$ of $(x(t),v(t),\eta(t))$ with respect to $W$ along the path $\zeta$. Observe that $A_{0,t} \zeta$ obeys the equation
\begin{align}
\tfrac{\d}{\d t} \pi_xA_{0,t} \zeta&= \pi_vA _{0,t}\zeta,\nt \\
\tfrac{\d}{\d t} \pi_v A_{0,t} \zeta&= -\pi_v A_{0,t} \zeta-\grad^2 U(x(t))\pi_x A_{0,t} \zeta-\grad^2 G(x(t))\pi_x A_{0,t} \zeta+\zeta(t)\nt \\
&\qquad\qquad-\int_0^\infty\close \grad H(x(t)-\eta(t;s))(\pi_x A _{0,t}\zeta-\pi_\eta A_{0,t} \zeta(s))K(s)\d s,\label{eqn:droplet:A}\\
\ddt \pi_\eta A_{0,t} \zeta&=-\partial_s\pi_\eta A_{0,t} \zeta,\quad \pi_\eta A_{0,t} \zeta(0)=\pi_x A_{0,t} \zeta,\, t>0, \nt
\end{align}
with the initial condition
\begin{align*}
A _{0,0}\zeta=0.
\end{align*}
\begin{remark} \label{rem:Malliavin}
  We recall that when $\zeta$ is an adapted process, the Malliavin derivative $A_{0,t}\zeta$ satisfies the identity
\begin{align} \label{eqn:Malliavin-int-by-part}
    \E\la D f(x(t),v(t),\eta(t)),A_{0,t}\zeta\ra_{\Xcal_{p_2}}= \E\left[ f(x(t),v(t),\eta(t))\int_0^t\la \zeta(r),\d W(r)\ra\right],
\end{align}
for every $f\in C_b^1(\Xcal_{p_2})$. We note that the above equation is the so-called Malliavin integration by parts where the stochastic integral on the right-hand side is understood in the usual It\^o sense. In case $\zeta$ is non-adapted, then one interprets this as a Skorokhod integral. See \cite[Equation (3.2)]{hairer2011theory} and \cite[Chapter 1.3]{nualart2013malliavin} for further discussions of this point. Later in the proof of Lemma \ref{lem:droplet:asymptotic-Feller}, we will carefully pick $\zeta$ to prove that $P_t$ satisfies the asymptotic Feller property.
\end{remark}

Now, denoting
\begin{align} \label{form:rho_t}
\rho_t =J_{0,t} \xi-A_{0,t} \zeta,
\end{align}
we observe that $\rho_t$ solves 
\begin{align}
\tfrac{\d}{\d t} \pi_x \rho_t &= \pi_v \rho_t ,\qquad\qquad \rho_0 =\xi, \nt \\
\tfrac{\d}{\d t} \pi_v \rho_t &= -\pi_v \rho_t -\grad^2 U(x(t))\pi_x \rho_t -\grad^2 G(x(t))\pi_x \rho_t - \zeta(t) \nt \\
&\qquad\qquad-\int_0^\infty\close \grad H(x(t)-\eta(t;s))(\pi_x\rho_t -\pi_\eta\rho_t (s) )K(s)\d s,\label{eqn:droplet:rho}\\
\ddt \pi_\eta \rho_t &=-\partial_s\pi_\eta \rho_t ,\quad \pi_\eta \rho_t (0)=\pi_x \rho_t ,\, t>0. \nt 
\end{align}
Having introduced the needed processes, we now state and prove the required asymptotic strong Feller property in Lemma~\ref{lem:droplet:asymptotic-Feller}.

\begin{lemma} \label{lem:droplet:asymptotic-Feller}There exist positive constants $c$ and $C$ such that for all $(x,v,\eta)\in\Xcal_{p_2}$ and $f\in C^1_b(\Xcal_{p_2})$, $C^1_b(\Xcal_{p_2})$ being the space of bounded differentiable functions on $\Xcal_{p_2}$, the following holds:
\begin{align} \label{ineq:droplet:asymptotic-strong-Feller}
&\|DP_t f(x,v,\eta)\|_{\Xcal_{p_2}^*}   \nt \\
&\le Ce^{-ct}\sqrt{P_t \|Df\|^2_{\Xcal_{p_2}^*}(x,v,\eta)}+ C e^{\frac{1}{2}\Psi(x,v,\eta)}\sqrt{P_t|f|^2(x,v,\eta)},\quad t\ge 0,
\end{align}
where $\Psi(x,v,\eta)$ is as in~\eqref{form:Psi}.
\end{lemma}
\begin{proof}
Let $(x(t),v(t),\eta(t))$ be the solution of \eqref{eqn:droplet} with initial condition $(x,v,\eta)\in \Xcal_{p_2}$. For $f\in C^1_b(\Xcal_{p_2})$, a unit element $\xi\in \Xcal_{p_2}$ and a path $\zeta\in L^2(0,t;\rbb^d)$, from the relation \eqref{form:rho_t}, we have the following chain of implications:
\begin{align*}
\la DP_tf(x,v,\eta),\xi\ra_{\Xcal_{p_2}} &=\E\la D f(x(t),v(t),\eta(t)),J_{0,t}\xi\ra_{\Xcal_{p_2}}\\
&=\E\la D f(x(t),v(t),\eta(t)),\rho_t\ra_{\Xcal_{p_2}}+ \E\la D f(x(t),v(t),\eta(t)),A_{0,t}\zeta\ra_{\Xcal_{p_2}}\\
&=\E\la D f(x(t),v(t),\eta(t)),\rho_t\ra_{\Xcal_{p_2}}+\E\left[ f(x(t),v(t),\eta(t))\int_0^t\la \zeta(r),\d W(r)\ra\right],
\end{align*}
where the last implication follows from identity \eqref{eqn:Malliavin-int-by-part}. In order to establish \eqref{ineq:droplet:asymptotic-strong-Feller}, it suffices to determine a suitable path $\zeta$ such that
\begin{align}  \label{ineq:droplet:asymptotic-strong-Feller:a}
\E\la D f(x(t),v(t),\eta(t)),\rho_t\ra_{\Xcal_{p_2}} \le Ce^{-ct}\sqrt{P_t\|Df\|^2_{\Xcal_{p_2}^*}(x,v,\eta)},
\end{align}
and that
\begin{align} \label{ineq:droplet:asymptotic-strong-Feller:b}
\E\left[ f(x(t),v(t),\eta(t))\int_0^t\la \zeta(r),\d W(r)\ra\right] \le C e^{\frac{1}{2}[\Phi(x,v)+\|\eta\|^{p_2}_{p_2}]}\sqrt{P_t|f|^2(x,v,\eta)},
\end{align}
for some positive constants $C,c$ independent of $(x,v,\eta)$, $t$, $f$ and $\xi$. To this end, we let $\alpha>0$ be an arbitrary positive constant. In view of \eqref{eqn:droplet:rho}, we choose the control path $\zeta$ as follows:
\begin{align} \label{form:zeta}
\zeta(t)& = -\pi_v\rho_t+5\alpha \pi_v\rho_t +6\alpha^2 \pi_x\rho_t-\grad^2 U(x(t))\pi_x \rho_t -\grad^2 G(x(t))\pi_x \rho_t  \nt \\
&\qquad\qquad-\int_0^\infty\close \grad H(x(t)-\eta(t;s))(\pi_x\rho_t -\pi_\eta\rho_t (s) )K(s)\d s.
\end{align}
With this choice of $\zeta$, observe that \eqref{eqn:droplet:rho} is reduced to 
\begin{align}
\tfrac{\d}{\d t} \pi_x \rho_t &= \pi_v \rho_t ,\qquad\qquad \rho_0 =\xi, \nt \\
\tfrac{\d}{\d t} \pi_v \rho_t &= -5\alpha \pi_v \rho_t-6\alpha^2 \pi_x\rho_t ,\label{eqn:droplet:rho:simplified}\\
\ddt \pi_\eta \rho_t &=-\partial_s\pi_\eta \rho_t ,\quad \pi_\eta \rho_t (0)=\pi_x \rho_t ,\, t>0. \nt 
\end{align}
On the one hand, by a routine calculation, the first two equations of \eqref{eqn:droplet:rho:simplified} yield the solution
\begin{align*}
\begin{pmatrix}
\pi_x \rho_t\\ \pi_v\rho_t
\end{pmatrix} = \left(3\pi_x \xi +\frac{1}{\alpha}\pi_v\xi\right) e^{-2\alpha t}\begin{pmatrix}1\\-2\alpha
\end{pmatrix} - \left(2\pi_x \xi +\frac{1}{\alpha}\pi_v\xi\right) e^{-3\alpha t}\begin{pmatrix}1\\-3\alpha
\end{pmatrix}.
\end{align*}
In particular, this implies the decay rate (recalling  $\xi$ is a unit element in $\Xcal_{p_2}$)
\begin{align} \label{ineq:|pi_x.rho_t|+|pi_v.rho_t|}
|\pi_x \rho_t|+ |\pi_v \rho_t| \le C e^{-2\alpha t} \big(|\pi_x \xi|+ |\pi_v \xi|\big)\le C\, e^{-2\alpha t},\quad t\ge 0,
\end{align}
for some positive constant $C$ independent of $t$ and $\xi$. On the other hand, concerning the last equation of \eqref{eqn:droplet:rho:simplified}, we employ an argument similar to \eqref{ineq:L.|eta|_(p_2)^(p_2)} and obtain
\begin{align*}
\ddt \|\pi_\eta\rho_t\|^{p_2}_{p_2}\le |\pi_x\rho_t|^{p_2}K(0)-\delta \|\pi_\eta\rho_t\|^{p_2}_{p_2},
\end{align*}
whence
\begin{align*}
\|\pi_\eta\rho_t\|^{p_2}_{p_2} & \le e^{-\delta t}\|\pi_\eta \xi\|^{p_2}_{p_2}+K(0)\int_0^t e^{-\delta (t-r)}|\pi_x\rho_r|^{p_2}\d r.
\end{align*}
Together with \eqref{ineq:|pi_x.rho_t|+|pi_v.rho_t|}, we infer the existence of positive constants $C$ and $c$ such that 
\begin{align} \label{ineq:|pi_eta.rho_t|}
\|\pi_\eta\rho_t\|^{p_2}_{p_2}  \le Ce^{-ct}\big(|\pi_x \xi|+ |\pi_v \xi|+\|\pi_\eta \xi\|_{p_2}\big)^{p_2}=Ce^{-ct}.
\end{align}
Turning back to \eqref{ineq:droplet:asymptotic-strong-Feller:a}, by the Cauchy-Schwarz inequality, we have
\begin{align*}
\E\la D f(x(t),v(t),\eta(t)),\rho_t\ra_{\Xcal_{p_2}} & \le \sqrt{\E \|D f(x(t),v(t),\eta(t))\|^2_{\Xcal_{p_2}^*} }\cdot \sqrt{\E\|\rho_t\|^2_{\Xcal_{p_2}}}\\
&\le \sqrt{P_t \|Df\|^2_{\Xcal_{p_2}^*}(x,v,\eta)}\cdot Ce^{-ct} ,
\end{align*}
where, in the last implication, we invoked \eqref{ineq:|pi_x.rho_t|+|pi_v.rho_t|} and \eqref{ineq:|pi_eta.rho_t|}. This establishes \eqref{ineq:droplet:asymptotic-strong-Feller:a}.

With regard to \eqref{ineq:droplet:asymptotic-strong-Feller:b}, we employ the Cauchy-Schwarz inequality with It\^o's isometry to see that
\begin{align} \label{ineq:droplet:asymptotic-strong-Feller:c}
\E\left[ f(x(t),v(t),\eta(t))\int_0^t\la \zeta(r),\d W(r)\ra\right] & \le \sqrt{\E|f(x(t),v(t),\eta(t))|^2}\cdot\sqrt{\E\int_0^\infty\close |\zeta(r)|^2\d r}.
\end{align}
To estimate the integral involving $\zeta$ on the above right-hand side, we recall the choice of $\zeta$ in~\eqref{form:zeta} and observe that
\begin{align}
|\zeta(t)|^2 &\le C\left[ |\pi_v\rho_t|^2+|\pi_x\rho_t|^2+\big(|\grad^2 U(x(t))|^2+|\grad^2 G(x(t))|^2\big)|\pi_x\rho_t|^2\right] \notag\\
&\qquad+C\Big|\int_0^\infty\close \grad H(x(t)-\eta(t;s))(\pi_x\rho_t-\pi_\eta\rho_t (s) )K(s)\d s\Big|^2.\label{ineq:droplet:asymptotic-strong-Feller:zeta} 
\end{align}
Concerning the term involving $\grad^2 U$ on the right-hand side of \eqref{ineq:droplet:asymptotic-strong-Feller:zeta}, from condition \eqref{cond:U(x)<x^q}, condition \eqref{cond:U(x)<x^q-2} and the choice of the function $\Phi$ in \eqref{form:Phi}, 
\begin{align*}
|\grad^2 U(x)|^2\le C (|U(x)|^2+1) \le Ce^{\Phi(x,v)}.
\end{align*}
Together with \eqref{ineq:E.e^Phi(t)<C.e^(-ct).e^Phi(0)+C} and \eqref{ineq:|pi_x.rho_t|+|pi_v.rho_t|}, we see that
\begin{align}
\E\int_0^\infty\close |\grad^2 U(x(r))|^2\cdot|\pi_x\rho_r |^2\d r 
&\le C\int_0^\infty\close e^{-2\alpha r}\, \E e^{\Phi(r)}  \d r \le C e^{\Phi(x,v)+\|\eta\|^{p_2}_{p_2}}.\label{ineq:droplet:asymptotic-strong-Feller:|grad.U|^2}
\end{align}
To estimate the term involving $\grad^2 G$ on the right-hand side of \eqref{ineq:droplet:asymptotic-strong-Feller:zeta}, we combine Eqs.~\eqref{cond:G:e^G>grad^2.G},~\eqref{ineq:int.e^(-alpha.t+G)/|x|} and \eqref{ineq:|pi_x.rho_t|+|pi_v.rho_t|} to deduce that %we recall condition \eqref{cond:G:e^G>grad^2.G} that
%\begin{align*}
%|\grad^2 G(x)|^2\le  C\frac{e^{G(x)}}{|x|^{\beta_1}}+C.
%\end{align*}
%Combining with \eqref{ineq:int.e^(-alpha.t+G)/|x|} and \eqref{ineq:|pi_x.rho_t|+|pi_v.rho_t|}, we deduce
\begin{align}
&\E\int_0^\infty\close |\grad^2 G(x(r))|^2\cdot|\pi_x\rho_r |^2\d r \nt \\
 & \le C\int_0^\infty\close  e^{-2\alpha r}\, \E \left[\frac{e^{G(x(r))}}{|x(r)|^{\beta_1}}\right] \d r+ C\int_0^\infty\close  e^{-2\alpha r}\d r \le C e^{\Phi(x,v)+\|\eta\|^{p_2}_{p_2}}.\label{ineq:droplet:asymptotic-strong-Feller:|grad.G|^2}
\end{align}
Next, considering the integral on the right-hand side of~\eqref{ineq:droplet:asymptotic-strong-Feller:zeta}, we employ condition~\eqref{cond:H:1} to see that
\begin{align}
&\Big|\int_0^\infty\close \grad H(x(t)-\eta(t;s))(\pi_x\rho_t -\pi_\eta\rho_t (s) )K(s)\d s\Big|^2  \nt \\
&\le C\Big|\int_0^\infty\close\big(|x(t)|^{p_1}+|\eta(t;s)|^{p_1}+1\big)\cdot \big( |\pi_x\rho_t |+|\pi_\eta\rho_t (s)|\big) K(s)\d s\Big|^2 \nt \\
&\le C\Big(|x(t)|^{2p_1}|\pi_x\rho_t|^2+|\pi_x\rho_t|^2\Big|\int_0^\infty\close|\eta(t;s)|^{p_1}K(s)\d s\Big|^2 +|\pi_x\rho_t|^2 \nt  \\
&\qquad\qquad + \big(|x(t)|^{2p_1}+1\big)\Big|\int_0^\infty\close |\pi_\eta\rho_t(s)|K(s)\d s\Big|^2+ \Big|\int_0^\infty\close |\eta(t;s)|^{p_1}|\pi_\eta\rho_t(s)|K(s)\d s\Big|^2  \Big). \label{ineq:droplet:asymptotic-strong-Feller:H}
\end{align}
Recalling the choice of $p_2$ as in \eqref{cond:p_2}, % that
%\begin{align*}
%p_2> \max\{2p_1,p_1+1\},
%\end{align*}
we have
\begin{align*}
\Big|\int_0^\infty\close|\eta(t;s)|^{p_1}K(s)\d s\Big|^2 & \le \int_0^\infty\close |\eta(t;s)|^{2p_1}K(s)\d s \int_0^\infty\close K(s)\d s \le C\|\eta(t)\|^{p_2}_{p_2}+C,
\end{align*}
and that
\begin{align*}
\Big|\int_0^\infty\close |\pi_\eta\rho_t(s)|K(s)\d s\Big|^2  & \le \Big[\int_0^\infty\close |\pi_\eta\rho_t(s)|^{p_2}K(s)\d s\Big]^\frac{2}{p_2} \cdot \Big[\int_0^\infty\close K(s)\d s\Big]^\frac{2}{p_2^*} = C\|\pi_\eta\rho_t\|^{2}_{p_2},
\end{align*}
where $p_2^*$ denotes the conjugate of $p_2$, i.e., $\frac{1}{p_2}+\frac{1}{p_2^*}=1$. Also,
\begin{align*}
\Big|\int_0^\infty\close |\eta(t;s)|^{p_1}|\pi_\eta\rho_t(s)|K(s)\d s\Big|^2 & \le \Big|\int_0^\infty\close |\eta(t;s)|^{p_1 p_2^*}K(s)\d s\Big|^{\frac{2}{p_2^*}}\, \Big| \int_0^\infty \close |\pi_\eta\rho_t(s)|^{p_2}K(s)\d s\Big|^{\frac{2}{p_2}}.
\end{align*}
Since $p_2>p_1+1$, note that $p_1p_2^* <p_2$. As a consequence,
\begin{align*}
&\Big|\int_0^\infty\close |\eta(t;s)|^{p_1 p_2^*}K(s)\d s\Big|^{\frac{2}{p_2^*}} \, \Big|\int_0^\infty \close |\pi_\eta\rho_t(s)|^{p_2}K(s)\d s\Big|^{\frac{2}{p_2}} \\
&\le C \|\eta(t)\|^{\frac{2p_2}{p_2^*}}_{p_2} \|\pi_\eta\rho_t\|^2_{p_2}+C\|\pi_\eta\rho_t\|^2_{p_2}. 
\end{align*}
Combining with \eqref{ineq:droplet:asymptotic-strong-Feller:H}, we deduce the bound
\begin{align*}
&\E \Big|\int_0^\infty\close \grad H(x(t)-\eta(t;s))(\pi_x\rho_t -\pi_\eta\rho_t (s) )K(s)\d s\Big|^2 \\
&\le C |\pi_x\rho_t|^2\E\big[ |x(t)|^{2p_1}+\|\eta(t)\|^{p_2}_{p_2}+1\big] + C\|\pi_\eta\rho_t\|^{2}_{p_2}\E\big[ |x(t)|^{2p_1}+\|\eta(t)\|^{\frac{2p_2}{p_2^*}}_{p_2} +1\big].
\end{align*}
Since $U(x)$ dominates $|x|^{q_0}$, cf. \eqref{cond:U(x)<x^q}, we have
\begin{align*}
|x|^{2p_1}+\|\eta\|^{p_2}_{p_2}+\|\eta\|^{\frac{2p_2}{p_2^*}}_{p_2}+1\le C e^{\Phi(x,v)+\|\eta\|^{p_2}_{p_2}}.
\end{align*}
It follows that
\begin{align*}
&\E \Big|\int_0^\infty\close \grad H(x(t)-\eta(t;s))(\pi_x\rho_t -\pi_\eta\rho_t (s) )K(s)\d s\Big|^2 \\
&\le C \big(|\pi_x\rho_t|^2 +  \|\pi_\eta\rho_t\|^2_{p_2}\big)\E e^{\Phi(x(t),v(t))+\|\eta(t)\|^{p_2}_{p_2}}.
\end{align*}
In light of \eqref{ineq:E.e^Phi(t)<C.e^(-ct).e^Phi(0)+C}, we deduce further that
\begin{align*}
&\E \Big|\int_0^\infty\close \grad H(x(t)-\eta(t;s))(\pi_x\rho_t -\pi_\eta\rho_t (s) )K(s)\d s\Big|^2 \\
&\le C \big(|\pi_x\rho_t|^2 +  \|\pi_\eta\rho_t\|^2_{p_2}\big) e^{\Phi(x,v)+\|\eta\|^{p_2}_{p_2}},
\end{align*}
whence
\begin{align} \label{ineq:droplet:asymptotic-strong-Feller:H:a}
&\int_0^\infty\close \E \Big|\int_0^\infty\close \grad H(x(t)-\eta(t;s))(\pi_x\rho_t -\pi_\eta\rho_t (s) )K(s)\d s\Big|^2\d t \nt \\
&\le C e^{\Phi(x,v)+\|\eta\|^{p_2}_{p_2}}\int_0^\infty\big(|\pi_x\rho_t|^2 +  \|\pi_\eta\rho_t\|^2_{p_2}\big) \d t \le C e^{\Phi(x,v)+\|\eta\|^{p_2}_{p_2}},
\end{align}
where the last implication follows from \eqref{ineq:|pi_x.rho_t|+|pi_v.rho_t|} and \eqref{ineq:|pi_eta.rho_t|}.

Now, we collect~\eqref{ineq:droplet:asymptotic-strong-Feller:zeta}, \eqref{ineq:droplet:asymptotic-strong-Feller:|grad.U|^2}, \eqref{ineq:droplet:asymptotic-strong-Feller:|grad.G|^2} and \eqref{ineq:droplet:asymptotic-strong-Feller:H:a} while making use of~\eqref{ineq:|pi_x.rho_t|+|pi_v.rho_t|} again to arrive at the bound
\begin{align*}
\E\int_0^\infty\close |\zeta(t)|^2\d t &\le C (1+ e^{\Phi(x,v)+\|\eta\|^{p_2}_{p_2}})\le C e^{\Phi(x,v)+\|\eta\|^{p_2}_{p_2}}.
\end{align*}
This together with~\eqref{ineq:droplet:asymptotic-strong-Feller:c} implies that
\begin{align*} 
\E \Big[f(x(t),v(t),\eta(t))\int_0^t\la \zeta(r),\d W(r)\ra\Big]\le C e^{\frac{1}{2}[\Phi(x,v)+\|\eta\|^{p_2}_{p_2}]}\sqrt{P_t |f|^2(x,v,\eta)},
\end{align*}
which produces \eqref{ineq:droplet:asymptotic-strong-Feller:b}. The proof is thus finished.

\end{proof}

\subsection{Irreducibility} \label{sec:ergodicity:irreducibility}

For $R>0$, let $A_{R}$ be the bounded set in $\Xcal_{p_2}$ given by
\begin{align} \label{form:A_R}
A_R= \{(x,v,\eta)\in\Xcal_{p_2}: |x|+|x|^{-1}+|v|+\|\eta\|_{p_2}\le R\}.
\end{align}
Also, letting $e_1=(1,0,\dots,0)$ be a unit vector in $\rbb^d$, for $r\in (0,1)$ we denote by $B_r$ the disk centered at $(e_1,0,0)\in\Xcal_{p_2}$ with radius $r$, i.e.,
\begin{align} \label{form:B_r}
B_r= \{(x,v,\eta)\in\Xcal_{p_2}: |x-e_1|+|v|+\|\eta\|_{p_2}\le r\}.
\end{align}

In Proposition \ref{prop:irreducible} below, we assert that the probability of the solutions eventually entering $B_r$ is uniform with respect to initial data in $A_R$. In turn, this will be employed to establish the $\vr_R$-small property of $A_R$ in Section \ref{sec:ergodicity:proof-of-theorem}.

\begin{proposition} \label{prop:irreducible}

Given every initial condition $(x,v,\eta)\in\Xcal_{p_2}$, let $(x(t),v(t),\eta(t)$ be the solution of \eqref{eqn:droplet}. Then, for all positive constants $R>2,\,r\in (0,1)$, there exists $T=T(R,r)>0$ such that
\begin{align} \label{ineq:irreducible}
\inf_{(x,v,\eta)\in A_R} \P\big(  \big(x(T),v(T),\eta(T)\big) \in B_r \big)> 0.
\end{align}
where $A_R$ and $B_r$ are defined in \eqref{form:A_R} and \eqref{form:B_r}, respectively.

\end{proposition}

Owing to the combination of nonlinearities and memory in \eqref{eqn:droplet}, we will not directly establish \eqref{ineq:irreducible}. Instead, we will compare the dynamics \eqref{eqn:droplet} with the following simpler system:
\begin{subequations}\label{eqn:Langevin:droplet}
\begin{align}
\d x_1(t)&=v_1(t)\d t, \label{eqn:Langevin:droplet:x}  \\
\d v_1(t)& = -v_1(t)\d t-\grad U(x_1(t))\d t-\grad G(x_1(t))\d t+\d W(t),\label{eqn:Langevin:droplet:v}\\
\d \eta_1(t)&=-\partial_s \eta_1(t)\d t,\quad \eta_1(t;0)=x_1(t),\, t>0.\label{eqn:Langevin:droplet:eta}
\end{align}
\end{subequations}
We note that the first two equations \eqref{eqn:Langevin:droplet:x}-\eqref{eqn:Langevin:droplet:v} are the classical Langevin dynamics, which is decoupled from the $\eta_1-$equation. The irreducibility proof of Proposition \ref{prop:irreducible} essentially consists of two main steps. Firstly, we prove an analogous result for \eqref{eqn:Langevin:droplet}, while making use of the fact that the Langevin dynamics \eqref{eqn:Langevin:droplet:x}-\eqref{eqn:Langevin:droplet:v} itself is well-behaved. This is summarized in Lemma \ref{lem:irreducible:Langevin} below. Then, we will employ Girsanov's Theorem to show that the law of \eqref{eqn:droplet} is equivalent to \eqref{eqn:Langevin:droplet}, allowing us to deduce \eqref{ineq:irreducible}.

\begin{lemma} \label{lem:irreducible:Langevin}
For every initial condition $(x,v,\eta)\in\Xcal_{p_2}$, let $(x_1(t),v_1(t),\eta_1(t)$ be the solution of \eqref{eqn:Langevin:droplet}. Then, for all positive constants $R>2,\,r\in(0,1)$, there exists $T=T(R,r)>0$ such that for all $t\ge T$,
\begin{align} \label{ineq:irreducible:Langevin:(x,v,eta)}
\inf_{(x,v,\eta)\in A_R} \P\big(  \big(x_1(t),v_1(t),\eta_1(t)\big) \in B_r \big)> 0,
\end{align}
where $A_R$ and $B_r$ are defined in \eqref{form:A_R} and \eqref{form:B_r}, respectively.

\end{lemma}

The proof of Lemma \ref{lem:irreducible:Langevin} will be presented at the end of this subsection. We now invoke Lemma \ref{lem:irreducible:Langevin} to conclude Proposition \ref{prop:irreducible}, whose argument is standard and can be found in the literature \cite{foldes2019large,nguyen2023small,
seong2023exponential}.

\begin{proof}[Proof of Proposition \ref{prop:irreducible}]
We first rewrite the auxiliary system \eqref{eqn:Langevin:droplet} as follows:
\begin{align*}
\d x_1(t)&=v_1(t)\d t,  \\
\d v_1(t)& = -v_1(t)\d t-\grad U(x_1(t))\d t-\grad G(x_1(t))\d t-\int_0^\infty H(x_1(t)-\eta_1(t;s))K(s)\d s \d t +\d W(t), \\
&\qquad\qquad +\int_0^\infty H(x_1(t)-\eta_1(t;s))K(s)\d s \d t,\\
\d \eta_1(t)&=-\partial_s \eta_1(t)\d t,\quad \eta_1(t;0)=x_1(t),\, t>0.
\end{align*}
Observe that the above system only differs from \eqref{eqn:droplet} by the appearance of the term $+\int_0^\infty H(x_1(t)-\eta_1(t;s))K(s)\d s \d t$. In order to produce the uniform bound \eqref{ineq:irreducible} by making use of \eqref{ineq:irreducible:Langevin:(x,v,eta)}, we first show that $\big(x(\cdot),v(\cdot),\eta(\cdot)\big)$ and $\big(x_1(\cdot),v_1(\cdot),\eta_1(\cdot)\big)$ are equivalent in law on $C([0,t];\Xcal_{p_2})$. In turn, this can be established by Girsanov's Theorem provided the following Novikov's condition is verified:
\begin{align} \label{cond:Novikov}
\E\exp\Big\{\frac{1}{2}\int_0^t \Big|\int_0^\infty H\big(x_1(r)-\eta_1(r;s)\big) K(s)\d s\Big|^2   \d r \Big\}<\infty.
\end{align} 
To establish~\eqref{cond:Novikov}, we recall condition \eqref{cond:H:1} and condition \eqref{cond:p_2} to estimate
\begin{align*}
\Big|\int_0^\infty H\big(x_1(r)-\eta_1(r;s)\big) K(s)\d s\Big|^2& \le c\Big(|x_1(r)|^{p_1}+ \int_0^\infty \big|\eta_1(r;s)\big|^{p_1}  K(s)\d s+1\Big)^2\\
&\le c|x_1(r)|^{2p_1}+c \|\eta_1(r)\|^{2p_1}_{2p_1}+c\\
&\le c |x_1(r)|^{p_2}+c \|\eta_1(r)\|^{p_2}_{p_2}+c.
\end{align*}
To further bound the above right-hand side, we employ the argument as in \eqref{ineq:L.|eta|_(p_2)^(p_2)} while making use of condition \eqref{cond:K:1} to see that
\begin{align*}
\ddt \|\eta_1(t)\|^{p_2}_{p_2} \le K(0)|\eta_1(t;0)|^{p_2}-\delta \|\eta_1(t)\|^{p_2}_{p_2},
\end{align*}
whence (recalling $\eta_1(t;0)=x_1(t)$)
\begin{align} \label{ineq:eta_1(t)}
\|\eta_1(t)\|^{p_2}_{p_2} \le e^{-\delta t} \|\eta\|_{p_2}^{p_2}+K(0)\int_0^t |x_1(s)|^{p_2}\d s.
\end{align}
As a consequence, for all $r\in[0,t]$, it holds that
\begin{align}
\Big|\int_0^\infty H\big(x_1(r)-\eta_1(r;s)\big) K(s)\d s\Big|^2& \le c|x_1(r)|^{p_2}+c \|\eta_1(r)\|^{p_2}_{p_2}+c \nt \\
&\le  c|x_1(r)|^{p_2}+  \|\eta\|_{p_2}^{p_2}+c\int_0^t |x_1(s)|^{p_2}\d s+c  \nt \\
&\le c\|\eta\|_{p_2}^{p_2} +c \sup_{r\in[0,t]}|x_1(r)|^{p_2}+c, \label{ineq:int.H(x_1-eta_1).K.ds}
\end{align}
for some positive constant $c$ that possibly depends on time $t$. Next, since $U(x)$ dominates $|x|^{p_2}$, cf. conditions \eqref{cond:q_0} and~\eqref{cond:p_2}, we infer that, for all $\beta$ sufficiently small,
\begin{align*}
c \sup_{r\in[0,t]}|x_1(r)|^{p_2} \le \beta  \sup_{r\in[0,t]}U(x_1(r))+C\le \beta  \sup_{r\in[0,t]}\Phi(x_1(r),v_1(r))+C,
\end{align*}
where the last implication follows from the definition of $\Phi$ in \eqref{form:Phi}. Together with \eqref{ineq:int.H(x_1-eta_1).K.ds}, we obtain the a.s. bound
\begin{align*} 
 \int_0^t \Big|\int_0^\infty H\big(x_1(r)-\eta_1(r;s)\big) K(s)\d s\Big|^2   \d r \le  c \|\eta\|_{p_2}^{p_2}+ \beta \sup_{r\in[0,t]}\Phi(x_1(r),v_1(r))+C.
\end{align*}
In view of Lemma \ref{lem:Langevin:E.exp(sup.Phi(x,v))}, we deduce that
\begin{align}
&\E\exp\Big\{\frac{1}{2}\int_0^t \Big|\int_0^\infty H\big(x_1(r)-\eta_1(r;s)\big) K(s)\d s\Big|^2   \d r \Big\}  \nt \\
&\le \exp \big\{c\|\eta\|_{p_2}^{p_2}+C\big\} \cdot \E \exp \Big\{ \beta  \sup_{r\in[0,t]}\Phi(x_1(r),v_1(r))  \Big\} \le C(x,v,\eta,t)<\infty, \label{cond:Novikov:(x,v,eta)}
\end{align}
which verifies the Novikov condition \eqref{cond:Novikov}, thereby establishing the equivalence in law.

Turning back to \eqref{ineq:irreducible}, we introduce the change-of-measure function
\begin{align*}
Q(t) &= \exp\Big\{\int_0^t \Big\la \int_0^\infty H\big(x_1(r)-\eta_1(r;s)\big) K(s)\d s,\d W(r) \Big\ra\\
&\qquad\qquad\qquad\qquad -\frac{1}{2}\int_0^t \Big|\int_0^\infty H\big(x_1(r)-\eta_1(r;s)\big) K(s)\d s\Big|^2   \d r \Big\}.
\end{align*}
In light of Girsanov's Theorem, for all time $t>0$, the following holds:
\begin{align*}
\P\big(  \big(x(t),v(t),\eta(t)\big) \in B_r \big) = \E\big[ Q(t)\cdot \boldsymbol{1}\big\{  \big(x_1(t),v_1(t),\eta_1(t)\big) \in B_r \big\} \big],
\end{align*}
where $B_r$ is defined in \eqref{form:B_r}. On the one hand, letting $\lambda\in (0,1)$ be given and chosen later, we have
\begin{align*}
\P\big(  \big(x(t),v(t),\eta(t)\big) \in B_r \big) &= \E\big[ Q(t)\cdot \boldsymbol{1}\big\{  \big(x_1(t),v_1(t),\eta_1(t)\big) \in B_r \big\} \big]\\
&\ge  \lambda \P\big( Q(t)>\lambda, \big(x_1(t),v_1(t),\eta_1(t)\big) \in B_r \big).
\end{align*}
On the other hand, 
\begin{align*}
\P\big( \big(x_1(t),v_1(t),\eta_1(t)\big) \in B_r \big)  & \le \P\big( Q(t)>\lambda, \big(x_1(t),v_1(t),\eta_1(t)\big) \in B_r \big)+ \P\big(Q(t)\le \lambda\big).
\end{align*}
We deduce from the two estimates above that
\begin{align*}
\frac{1}{\lambda}\P\big(  \big(x(t),v(t),\eta(t)\big) \in B_r \big)  &\ge \P\big( \big(x_1(t),v_1(t),\eta_1(t)\big) \in B_r \big) -\P\big(Q(t)\le \lambda\big),
\end{align*}
whence (recalling the set $A_R$ in \eqref{form:A_R})
\begin{align}
&\frac{1}{\lambda}\inf_{(x,v,\eta)\in A_R}\P\big(  \big(x(t),v(t),\eta(t)\big) \in B_r \big)  \nt  \\
&\ge \inf_{(x,v,\eta)\in A_R}\P\big( \big(x_1(t),v_1(t),\eta_1(t)\big) \in B_r \big) -\sup_{(x,v,\eta)\in A_R}\P\big(Q(t)\le \lambda\big). \label{ineq:inf.P(x,v,eta)>inf.P(x_1,v_1,eta_1)-sup.P(Q<lambda)}
\end{align}
To estimate the supremum on the above right-hand side, we note that
\begin{align*}
\P\big(Q(t)\le \lambda\big) & = \P\Big( \frac{1}{2}\int_0^t \Big|\int_0^\infty H\big(x_1(r)-\eta_1(r;s)\big) K(s)\d s\Big|^2   \d r \\
&\qquad\qquad -\int_0^t \Big\la \int_0^\infty H\big(x_1(r)-\eta_1(r;s)\big) K(s)\d s,\d W(r) \Big\ra \ge -\log (\lambda)\Big).
\end{align*} 
By Markov's inequality, we deduce that
\begin{align*}
\P\big(Q(t)\le \lambda\big) & \le \frac{1}{-\log(\lambda)}\Big( \E \Big[\frac{1}{2}\int_0^t \Big|\int_0^\infty H\big(x_1(r)-\eta_1(r;s)\big) K(s)\d s\Big|^2   \d r\Big]\\
&\qquad\qquad\qquad+ \E\Big| \int_0^t \Big\la \int_0^\infty H\big(x_1(r)-\eta_1(r;s)\big) K(s)\d s,\d W(r) \Big\ra  \Big|\Big).
\end{align*}
Also, H\"{o}lder's inequality and It\^o's isometry yield
\begin{align*}
&\E\Big| \int_0^t \Big\la \int_0^\infty H\big(x_1(r)-\eta_1(r;s)\big) K(s)\d s,\d W(r) \Big\ra  \Big|\\
&\le \sqrt{t}\cdot \sqrt{\int_0^t \E\Big|\int_0^\infty H\big(x_1(r)-\eta_1(r;s)\big) K(s)\d s\big|^2\d r }.
\end{align*}
So, 
\begin{align*}
\P\big(Q(t)\le \lambda\big) & \le \frac{1}{-\log(\lambda)}\Big( \E \Big[\int_0^t \Big|\int_0^\infty H\big(x_1(r)-\eta_1(r;s)\big) K(s)\d s\Big|^2   \d r\Big]+t\Big)\\
&\le  \frac{1}{-\log(\lambda)}\cdot \exp\Big\{c\|\eta\|_{p_2}^{p_2}+\beta\Phi(x,v)+C \Big\},
\end{align*}
where the last implication follows from \eqref{cond:Novikov:(x,v,eta)} and Lemma \ref{lem:Langevin:E.exp(sup.Phi(x,v))}. Since $(x,v,\eta)\in A_R$, we infer that
\begin{align*}
\sup_{(x,v,\eta)\in A_R}\P\big(Q(t)\le \lambda\big) & \le \frac{C}{-\log(\lambda)},
\end{align*}
for some positive constant $C=C(R,t)$ independent of $\lambda$. From \eqref{ineq:inf.P(x,v,eta)>inf.P(x_1,v_1,eta_1)-sup.P(Q<lambda)} together with Lemma \ref{lem:irreducible:Langevin}, we obtain
\begin{align*}
&\frac{1}{\lambda}\inf_{(x,v,\eta)\in A_R}\P\big(  \big(x(t),v(t),\eta(t)\big) \in B_r \big)  \nt  \\
&\ge \inf_{(x,v,\eta)\in A_R}\P\big( \big(x_1(t),v_1(t),\eta_1(t)\big) \in B_r \big) -\sup_{(x,v,\eta)\in A_R}\P\big(Q(t)\le \lambda\big)\ge c-  \frac{C}{|\log(\lambda)|}.
\end{align*}
We emphasize again that the positive constants $c,C$ do not depend on the choice of $\lambda$. Lastly, by taking $\lambda$ small enough, we immediately obtain the uniform lower bound
\begin{align*}
\inf_{(x,v,\eta)\in A_R}\P\big(  \big(x(t),v(t),\eta(t)\big) \in B_r \big)\ge \lambda \Big(   c-  \frac{C}{|\log(\lambda)|} \Big) >0,
\end{align*}
thereby finishing the irreducibility argument.

\end{proof}

We now turn to Lemma \ref{lem:irreducible:Langevin}, whose argument will employ the Support Theorem \cite{stroock1972degenerate,stroock1972support} as well as Lemma \ref{lem:irreducible:control-xtilde}. 
%See also \cite{herzog2019ergodicity,herzog2021stability,
%lu2019geometric} for related results on the support of SDEs with singular potentials. 

\begin{proof}[Proof of Lemma \ref{lem:irreducible:Langevin}]
Firstly, we consider the pair $(x_1,v_1)$ solving \eqref{eqn:Langevin:droplet:x}-\eqref{eqn:Langevin:droplet:v}. To derive an estimate in probability for $(x_1,v_1)$, we introduce the deterministic control problem
\begin{align} \label{eqn:Langevin:droplet:xtile}  
\d \xt_1(t)&=\vt_1(t)\d t, \nt  \\
\d \vt_1(t)& = -\vt_1(t)\d t-\grad U(\xt_1(t))\d t-\grad G(\xt_1(t))\d t+\d \Gamma(t),\\
\xt_1(0)&=x,\quad \vt_1(0)=v,\nt 
\end{align}
where $\Gamma\in C^1([0,\infty);\rbb^d)$ is a control function. By the Stroock-Varadhan Support Theorem, for each $t,\,\varepsilon>0$ and $(x_0,v_0)$, there exists $\varepsilon_1=\varepsilon_1(t,\varepsilon)$ such that
\begin{align*}
\inf_{(x,v)\in E\big((x_0,v_0),\varepsilon_1\big)} \P\Big( \sup_{\ell\in [0,t]} |x_1(\ell)-\xt_1(\ell)|+ |v_1(\ell)-\vt_1(\ell)|\le \varepsilon \Big) >0,
\end{align*}
where
\begin{align*}
E\big((x_0,v_0),\varepsilon_1\big) = \big\{(x,v)\in \rbb^d\setminus\{0\}\times\rbb^d: |x-x_0|+|v-v_0|<\varepsilon_1\big\}.
\end{align*}
In view of Lemma \ref{lem:irreducible:control-xtilde}, for each $R>2,\varepsilon>0$ and $(x_0,v_0)$ such that $|x_0|+|x_0|^{-1}+|v_0|\le R$, there exists a path $(\xt_1,\vt_1,\Gamma)$ satisfying 
\begin{align*}
(\xt_1(t),\vt_1(t))=(e_1,0)\quad\text{and}\quad \int_0^t |\xt_1(s)|^{p_2}\d s < \varepsilon.
\end{align*}
It follows that
\begin{align*}
\inf_{(x,v)\in E\big((x_0,v_0),\varepsilon_1\big)} \P\Big(  |x_1(t)-e_1|+ |v_1(t)|\le \varepsilon,\quad \text{and}\quad \int_0^t |x_1(s)|^{p_2}\d s \le C\,t\,\varepsilon \Big) >0,
\end{align*}
for some positive constant $C$ independent of $t,\varepsilon,\varepsilon_1$ and $R$. Since the set $\{|x|+|x|^{-1}+|v|\le R\}$ is compact and thus can be covered by finitely many open disks $E((x_0,v_0),\varepsilon_1)$, we obtain the uniform bound
\begin{align} \label{ineq:inf_(x+x^(-1)+v<R)}
\inf_{(x,v):|x|+|x|^{-1}+|v|\le R} \P\Big(  |x_1(t)-e_1|+ |v_1(t)|\le \varepsilon,\quad \text{and}\quad \int_0^t |x_1(s)|^{p_2}\d s \le C\,t\,\varepsilon \Big) >0.
\end{align}

Turning to the $\eta_1$-equation \eqref{eqn:Langevin:droplet:eta}, we combine \eqref{ineq:eta_1(t)} and \eqref{ineq:inf_(x+x^(-1)+v<R)} to see that for all $R>2,t>1$ and $\varepsilon>0$, it holds that
\begin{align*}
\inf_{(x,v,\eta)\in A_R} \P\Big(  |x_1(t)-e_1|+ |v_1(t)|\le \varepsilon,\quad \text{and}\quad \|\eta_1(t)\|^{p_2}_{p_2} \le e^{-\delta t}R^{p_2} + C\,t\,\varepsilon \Big) >0.
\end{align*}
By taking $t$ sufficiently large and then shrinking $\varepsilon$ small enough, this produces the bound \eqref{ineq:irreducible:Langevin:(x,v,eta)}, as claimed.

\end{proof}

\subsection{Proof of the main results} \label{sec:ergodicity:proof-of-theorem}
In this section, we establish the exponential mixing rate of $P_t$ toward the unique invariant probability measure $\nu$ in terms of Wassertstein distances and observables, as presented in Theorem \ref{thm:geometric-ergodicity} and Corollary \ref{cor:geometric-ergodicity}, respectively. The argument follows the approach of \cite{butkovsky2020generalized,hairer2011asymptotic}, which consists of constructing Lyapunov functions, $\vr_N$-\emph{contracting} property, and $\vr_N$-\emph{small} sets where $\vr_N$ is the metric defined in \eqref{form:varrho_N}. For the reader's convenience, we recall these notions below.

\begin{definition} \label{def:Lyapunov-contracting-dsmall} 1. A function $V:\Xcal_{p_2}\to [0,\infty)$ is called a \emph{Lyapunov} function for $P_t$ if
\begin{align*}
P_tV(X)\le C e^{-ct}V(X)+C,\quad t\ge 0,\, X\in\Xcal_{p_2},
\end{align*}
for some positive constants $c,C$ independent of $X$ and $t$.

2. A distance-like function $\vr$ bounded by 1 is called \emph{contracting} for $P_t$ if there exists $\lambda_1\in(0,1)$ such that for any $X,\Xt$ with $\vr(X,\Xt)<1$, it holds that
\begin{align*}
\W_{\vr}\big( P_t(X,\cdot) ,P_t(\Xt,\cdot)  \big) \le \lambda_1 \vr(X,\Xt).
\end{align*}

3. A set $A \subset \Xcal_{p_2}$ is called $\vr$-\emph{small} for $P_t$ if for some $\lambda_2=\lambda_2(A)>0$,
\begin{align*}
\sup_{X,\Xt\in A}\W_{\vr}\big( P_t(X,\cdot) ,P_t(\Xt,\cdot)  \big) \le 1-\lambda_2.
\end{align*}

\end{definition}

On the one hand, the existence of a Lyapunov function is guaranteed by the energy estimates in Lemma \ref{lem:moment-bound}. On the other hand, in Proposition \ref{prop:contracting-dsmall} below, we will demonstrate that one can always tune $\vr_N$ defined in \eqref{form:varrho_N} such that $\vr_N$ is contracting for $P_t$ and that the bounded set $A_R$ defined in \eqref{form:A_R} is small with respect to $\vr_N$. 

\begin{proposition} \label{prop:contracting-dsmall}
Let $\vr_N$ and $A_R$ be defined in \eqref{form:varrho_N} and \eqref{form:A_R}, respectively. Then,

1. There exist $N_1$ and $t_1$ sufficiently large such that for all $N>N_1$ and $t>t_1$, $\vr_N$ is contracting for $P_t$ in the sense of Definition \ref{def:Lyapunov-contracting-dsmall}, part 2., with $\lambda_1=\frac{1}{2}$, i.e.,
\begin{align} \label{ineq:varrho_N:contracting}
\W_{\vr_N}\big( P_t(X,\cdot) ,P_t(\Xt,\cdot)  \big) \le \frac{1}{2} \vr_N(X,\Xt),
\end{align}
whenever $\vr_N(X,\Xt)<1$.

2. For all $N>0,R>2$, there exists $t_2=t_2(N,R)$ such that for all $t\ge t_2$, the set $A_R$ is $\vr_N$-small in the sense of Definition \ref{def:Lyapunov-contracting-dsmall}, part 3., with a constant $\lambda_2=\lambda_2(R)\in(0,1)$, i.e.,
\begin{align} \label{ineq:varrho_N:dsmall}
\sup_{X,\Xt\in A_R}\W_{\vr_N}\big( P_t(X,\cdot) ,P_t(\Xt,\cdot)  \big) \le 1-\lambda_2.
\end{align}
\end{proposition}

The proof of Proposition \ref{prop:contracting-dsmall} is quite standard \cite{hairer2011asymptotic} and makes use of the auxiliary results presented in Section \ref{sec:ergodicity:asymptotic-strong-feller} and Section \ref{sec:ergodicity:irreducibility}. For the sake of clarity, we will defer the explicit argument to the end of this section. We are now in a position to conclude the proof of Theorem \ref{thm:geometric-ergodicity} by verifying the conditions of \cite[Theorem 4.8]{hairer2011asymptotic}

\begin{proof}[Proof of Theorem \ref{thm:geometric-ergodicity}] On the one hand, from Lemma \ref{lem:moment-bound}, cf., \eqref{ineq:E.e^Phi(t)<C.e^(-ct).e^Phi(0)+C}, the function $V=e^{\Psi}$ plays the role of the Lyapunov function as in Definition \ref{def:Lyapunov-contracting-dsmall}, part 1. On the other hand, in view of Proposition \ref{prop:contracting-dsmall}, the distance $\vr_N$ defined in \eqref{form:varrho_N} is contracting for $P_t$, and any bounded set $A_R$ $(R>2)$ is $\vr_N$-small. In other words, all of the conditions of \cite[Theorem 4.8]{hairer2011asymptotic} are met. In light of \cite[Theorem 4.8]{hairer2011asymptotic}, we obtain the existence and uniqueness of the invariant probability measure $\nu$ for $P_t$ as well as the exponential convergence rate \eqref{ineq:geometric-ergodciity}.

\end{proof}

As a consequence of Theorem \ref{thm:geometric-ergodicity}, we deduce the mixing rate with respect to observables in Corollary \ref{cor:geometric-ergodicity}. Since the proof of Corollary \ref{cor:geometric-ergodicity} is short, we include it here for the sake of completeness.
\begin{proof}[Proof of Corollary \ref{cor:geometric-ergodicity}]
Let $f\in C(\Xcal_{p_2};\rbb)$ be such that $[f]_{\text{Lip},\vrt_N}<\infty$. Given $X_0\in \Xcal_{p_2}$, in view of estimate \eqref{ineq:W_d(nu_1,nu_2):dual}, it holds that
\begin{align*}
\W_{\vrt_N}\big(P_t\delta_{X_0},\nu \big)\ge \frac{1}{[f]_{\text{Lip},\vrt_N}} \Big|P_t f(X_0)-\int_{\Xcal_{p_2}}\close f(X)\nu(\d X) \Big|.
\end{align*}
From Theorem \ref{thm:geometric-ergodicity}, we deduce that
\begin{align*}
 \Big|P_t f(X_0)-\int_{\Xcal_{p_2}}\close f(X)\nu(\d X) \Big| \le  [f]_{\text{Lip},\vrt_N}\W_{\vrt_N}\big(\delta_{X_0},\nu \big) Ce^{-ct}.
\end{align*} 
This produces \eqref{ineq:geometric-ergodciity:P_tf}, thereby finishing the proof.

\end{proof}

We now turn to Proposition \ref{prop:contracting-dsmall}, which was ultimately invoked to conclude the main result in Theorem \ref{thm:geometric-ergodicity}. The proof of Proposition \ref{prop:contracting-dsmall} will employ the auxiliary estimates performed in Section \ref{sec:ergodicity:asymptotic-strong-feller} and Section \ref{sec:ergodicity:irreducibility}. See also \cite{butkovsky2020generalized, hairer2011asymptotic}.

\begin{proof}[Proof of Proposition \ref{prop:contracting-dsmall}]
1. Letting $N>0$ be given and chosen later, consider $X,\Xt\in\Xcal_{p_2}$ such that $\vr_N(X,\Xt)<1$. From the expression \eqref{form:varrho_N}, this implies that
\begin{align*}
\vr_N(X,\Xt) = N \vr(X,\Xt) \mi 1 = N \vr(X,\Xt),
\end{align*}
where $\vr$ is defined in \eqref{form:varrho}. By the dual Kantorovich formula \eqref{form:W_d:dual-Kantorovich}, the inequality \eqref{ineq:varrho_N:contracting} is equivalent to 
\begin{align} \label{ineq:varrho_N:contracting:P_t}
|P_t f(X)-P_tf(\Xt)|\le \frac{1}{2}N\vr(X,\Xt),
\end{align}
which holds for all $f\in C^1_b(\Xcal_{p_2})$ satisfying $[f]_{\text{Lip},\vr_N}\le 1$. Here, we recall $[f]_{\text{Lip},\vr_N}$ is the Lipschitz norm with respect to $\vr_N$ in \eqref{form:Lipschitz}. Note also that \eqref{form:Lipschitz} may be recast as
\begin{align*}
[f]_{\text{Lip},\vr_N} = \sup_{Y\neq \Yt} \frac{f(Y)-f(\Yt)}{\vr_N(Y,\Yt)} =  \sup_{Y\neq \Yt} \frac{[f(Y)-f(e_1,0,0)]-[f(\Yt)-f(e_1,0,0)]}{\vr_N(Y,\Yt)}.
\end{align*}
In the above, $e_1=(1,0,\dots,0)$ is the unit element in $\rbb^d$. It therefore suffices to prove \eqref{ineq:varrho_N:contracting:P_t} for those functions $f$ such that $f(e_1,0,0)=0$. In particular, this implies that $\|f\|_{\infty}\le 1$ since
\begin{align*}
|f(Y)|=|f(Y)-f(e_1,0,0)|\le \vr_N\big(Y,(e_1,0,0)\big)\le 1,\quad Y\in\Xcal_{p_2}.
\end{align*}
Furthermore, for $\Yt$ sufficiently close to $Y$ such that the path
\begin{align*}
\gamma_*(s) = s\cdot Y+(1-s)\cdot \Yt\quad \text{for}\quad 0\leq s\leq 1
\end{align*}
belongs to $\Xcal_{p_2}$, we observe that
\begin{align*}
|f(Y)-f(\Yt)| \le \vr_N(Y,\Yt) &\le N\vr(Y,\Yt) \\
&\le N\int_0^1 e^{\frac{1}{2}\Psi(\gamma_*(s))}\d s \|Y-\Yt \|_{\Xcal_{p_2}}.
\end{align*}
In the above, we recall $\Psi(x,v,\eta)=\Phi(x,v)+\|\eta\|_{p_2}^{p_2}$ defined in \eqref{form:Psi}. It follows that
\begin{align*}
\frac{|f(Y)-f(\Yt)|}{\|Y-\Yt \|_{\Xcal_{p_2}}} \le  N\int_0^1 e^{\frac{1}{2}\Psi(\gamma_*(s))}\d s \to N e^{\frac{1}{2}\Psi(Y)}
\end{align*}
as $\Yt\to Y$ in $\Xcal_{p_2}$ by virtue of the Dominated Convergence Theorem. As a consequence,
\begin{align*}
\|Df(Y)\|_{\Xcal_{p_2}^*} \le N e^{\frac{1}{2}\Psi(Y)},\quad Y\in\Xcal_{p_2}. 
\end{align*}
It follows that
\begin{align*}
P_t \|Df\|^2_{\Xcal_{p_2}^*}(x,v,\eta)  & = \E\|Df(x(t),v(t),\eta(t))\|^2_{\Xcal_{p_2}^*} \\
&\le N^2 \E e^{\Psi(x(t),v(t),\eta(t))}\\
&\le C\,N^2 e^{\Psi(x,v,\eta)},
\end{align*}
where in the last implication above, we invoke \eqref{ineq:E.e^Phi(t)<C.e^(-ct).e^Phi(0)+C} with some positive constant $C$ independent of $N$ and $(x,v,\eta)$. Together with the asymptotic strong Feller property \eqref{ineq:droplet:asymptotic-strong-Feller}, we deduce that
\begin{align*}
&\|DP_t f(x,v,\eta)\|_{\Xcal_{p_2}^*}   \nt \\
&\le Ce^{-ct}\sqrt{P_t \|Df\|^2_{\Xcal_{p_2}^*}(x,v,\eta)}+ C e^{\frac{1}{2}\Psi(x,v,\eta)}\sqrt{P_t|f|^2(x,v,\eta)}\\
&\le C e^{-ct}\cdot  N  e^{\frac{1}{2}\Psi(x,v,\eta)}+ C e^{\frac{1}{2}\Psi(x,v,\eta)}\sqrt{P_t|f|^2(x,v,\eta)}\\
&\le  N  e^{\frac{1}{2}\Psi(x,v,\eta)}\Big[Ce^{-ct} +\frac{C}{N}  \Big].
\end{align*} 
We emphasize again that on the above right-hand side, $C$ and $c$ are independent of $N$, $t$ and $(x,v,\eta)$. By choosing $t$ and $N$ sufficiently large such that
\begin{align*}
Ce^{-ct} +\frac{C}{N}  <\frac{1}{2},
\end{align*}
we obtain
\begin{align} \label{ineq:|DP_t.f(X)|<1/2.N.e^(Psi(X))}
\|DP_t f(x,v,\eta)\|_{\Xcal_{p_2}^*}   \le \frac{1}{2} N  e^{\frac{1}{2}\Psi(x,v,\eta)}.
\end{align}

Now, returning to \eqref{ineq:varrho_N:contracting:P_t}, letting $\gamma\in C([0,1];\Xcal_{p_2})$ be an arbitrary path connecting $X$ and $\Xt$, we recast the left-hand side of \eqref{ineq:varrho_N:contracting:P_t} as
\begin{align*}
P_tf(X)-P_tf(\Xt) = \int_0^1 \la DP_tf(\gamma(s)),\gamma'(s)\ra_{\Xcal_{p_2}}\d s.
\end{align*}
Using H\"{o}lder's inequality and \eqref{ineq:|DP_t.f(X)|<1/2.N.e^(Psi(X))}, we infer that
\begin{align*}
|P_tf(X)-P_tf(\Xt) | &\le \int_0^1 \|DP_tf(\gamma(s))\|_{\Xcal_{p_2}^*} \|\gamma'(s)\|_{\Xcal_{p_2}}\d s\\
& \le  \frac{1}{2}N\int_0^1 e^{\frac{1}{2}\Psi(\gamma(s))}\|\gamma'(s)\|_{\Xcal_{p_2}}\d s.
\end{align*}
Equation \eqref{ineq:varrho_N:contracting:P_t}  then follows from the definition of the distance $\varrho$ given in \eqref{form:varrho}. In turn, this establishes the contracting property \eqref{ineq:varrho_N:contracting}, as claimed, for all pairs $(X,\Xt)$ for which $\varrho_N(X,\Xt)<1$.

2. With regard to the $\vr_N$-small property \eqref{ineq:varrho_N:dsmall}, let $r\in(0,\frac{1}{2})$ be given and chosen later. Since the set $B_r$ defined in \eqref{form:B_r} is bounded away from the origin in $\Xcal_{p_2}$, we note that for every $Z_0,\Zt_0\in  B_r$, the path
\begin{align*}
\gamma_* = s\cdot Z_0+(1-s)\cdot \Zt_0\quad \text{for}\quad 0\leq s\leq 1
\end{align*}
is guaranteed to belong to $C^1([0,1];B_r)$. Indeed, by the triangle inequality,
\begin{align*}
&|s\cdot \pi_x Z_0+(1-s)\cdot \pi_x \Zt_0-e_1|+ |s\cdot \pi_v Z_0+(1-s)\cdot \pi_v \Zt_0|+\|s\cdot \pi_\eta Z_0 +(1-s)\cdot \pi_\eta \Zt_0\|_{p_2} \\
&\le s\big( |\pi_x Z_0-e_1 |+|\pi_v Z_0|+\|\pi_\eta Z_0\|_{p_2}\big) + (1-s)  \big(|\pi_x \Zt_0-e_1 |+|\pi_v \Zt_0|+\|\pi_\eta \Zt_0\|_{p_2}\big) \le r.
\end{align*}
 In particular, this implies that
\begin{align*} 
\vr(Z_0,\Zt_0) \le \int_0^1 e^{\frac{1}{2}\Psi(\gamma_*(s))}\|\gamma_*'(s)\|_{\Xcal_{p_2}}\d s\le \sup_{Z\in B_r}e^{\frac{1}{2}\Psi(Z)} \cdot \|Z_0-\Zt_0\|_{\Xcal_{p_2}} \le C r,
\end{align*}
for some positive constant $C$ that does not depend on $r\in (0,\frac{1}{2})$. In other words, 
\begin{align} \label{ineq:varrho(Z,Zt)}
\sup_{Z_0,\Zt_0\in B_r}\vr(Z_0,\Zt_0) \le  C r,\quad r\in \left(0,\frac{1}{2}\right).
\end{align}

Returning to \eqref{ineq:varrho_N:dsmall}, given $X,\Xt\in A_R$, we denote by $(Y,\Yt)$ a coupling of $\big( P_t(X,\cdot),P_t(\Xt,\cdot) \big)$ such that $Y$ and $\Yt$ are independent. By the definition \eqref{form:W_d} and the expression \eqref{form:varrho_N}, we have the following chain of implications:
\begin{align*}
&\W_{\vr_N}\big( P_t(X,\cdot) ,P_t(\Xt,\cdot)  \big)\\
& \le \E \vr_N(Y,\Yt)  = \E \big[N\vr(Y,\Yt) \mi 1 \big]\\
&\le \E\Big[\big( N\vr(Y,\Yt) \mi 1\big)\boldsymbol{1}\big\{ \{Y\in B_r\}\cap \{\Yt\in B_r\} \big\}\Big] \\
&\qquad+ \E\Big[\big( N\vr(Y,\Yt) \mi 1\big)\boldsymbol{1}\big\{ \{Y\notin B_r\}\cup \{\Yt\notin B_r\} \big\}\Big]\\
&\le N \sup_{Z_0,\Zt_0\in B_r}\vr(Z_0,\Zt_0) \cdot \P\big(\{Y\in B_r\}\cap \{\Yt\in B_r\}\big) + \P\big(\{Y\notin B_r\}\cup \{\Yt\notin B_r\}\big).
\end{align*}
We invoke \eqref{ineq:varrho(Z,Zt)} to further estimate
\begin{align*}
\W_{\vr_N}\big( P_t(X,\cdot) ,P_t(\Xt,\cdot)  \big)  \le N Cr\cdot \P\big(\{Y\in B_r\}\cap \{\Yt\in B_r\}\big)+ 1- \P\big(\{Y\in B_r\}\cap \{\Yt\in B_r\}\big).
\end{align*}
Since $N$ and $C$ are independent of $r\in(0,\frac{1}{2})$, we pick $r$ sufficiently small such that $NCr<\frac{1}{2}$ and obtain
\begin{align*}
\W_{\vr_N}\big( P_t(X,\cdot) ,P_t(\Xt,\cdot)  \big)  &\le 1-\frac{1}{2}\P\big(\{Y\in B_r\}\cap \{\Yt\in B_r\}\big)\\
&= 1- \frac{1}{2} \P\big(Y\in B_r\big) \P\big(\Yt\in B_r\big),
\end{align*}
where the last identity follows from the choice of $Y$ and $\Yt$ being independent. In light of the irreducibility condition \eqref{ineq:irreducible}, we infer the existence of a time $t_2=t_2(R,r)$ such that for all $t\ge t_2$
\begin{align*}
\inf_{X\in A_R}P_t(X,B_r)>c=c(t)>0,
\end{align*}
whence
\begin{align*}
\W_{\vr_N}\big( P_t(X,\cdot) ,P_t(\Xt,\cdot)  \big)  &\le 1- \frac{1}{2} \P\big(Y\in B_r\big) \P\big(\Yt\in B_r\big) \le 1-\frac{1}{2}c^2.
\end{align*}
This produces \eqref{ineq:varrho_N:dsmall}, thereby finishing the proof.

\end{proof}

\section{Numerical simulations of the invariant measure}\label{sec:numerics}

We conclude the paper by presenting numerical examples for the invariant measure obtained from the stroboscopic model in dimension $d=2$, specifically, for which the walker evolves under the influence of an attractive harmonic potential $U(x) = |x|^2/2$ and repulsive Coulomb potential $G(x) = -\alpha\log(|x|)$. The purpose of this section is to give the reader an intuition about the qualitative structure of the walker's invariant measure, the existence and uniqueness of which we established in the preceding sections by proving Theorem~\ref{thm:geometric-ergodicity}. To that end, we employ the functional forms of the pilot-wave force $H(x) = \mathrm{J}_1(|x|)x/|x|$ (c.f. Remark~\ref{rem:U}) and memory kernel $K(t) = \mathrm{e}^{-t}$ (c.f. Remark~\ref{rem:K}), fixing the dimensionless walker mass $m=1$ and noise strength $\sigma = 1$. The forms of $H(x)$ and $K(t)$ are motivated by prior work on the topic~\cite{Eddi2011a,Molacek2013b}, in which approximations were derived for the spatiotemporal evolution of the wave field generated by a walker bouncing on a vertically vibrating fluid bath. 

Equation~\eqref{eqn:droplet:original} is solved numerically using an Euler-Maruyama time-stepping scheme with time step $\Delta t = 2^{-6}$ up to a final time $t_{\text{max}} = 10^4$, with the integral computed using the trapezoidal rule. To initialize the simulations, we set the {\it initial past} corresponding to the stationary state $x(t)=(2\pi,0)$ for $t < 0$. We consider three different values of $\alpha$: $\alpha=1$, 3 and 5. Recall from Remark~\ref{rem:G} that $\alpha=3$ is the ``edge case": while we have strictly speaking proven Theorem~\ref{thm:geometric-ergodicity} under the assumption $\alpha \geq 3$, we expect that the theorem is true for all values of $\alpha$. Indeed, the case $\alpha=0$, in which there is no singular potential, was covered by our prior work~\cite{nguyen2024invariant}. 

The results of the numerical simulations are shown in Fig.~\ref{Fig:Numerics}. The first two columns show the numerical solutions, while the last column shows the walker's radial position probability density function $p(r)$, where $r = |x|$. Panels (a) through (d) correspond to $\alpha=1$, (e) through (h) to $\alpha=3$ and (i) through (l) to $\alpha=5$. It is evident from the leftmost panels that the walker executes an erratic trajectory in a roughly annular region in the plane, with the radius of the inner circle increasing with the strength $\alpha$ of the repulsive potential. In all three cases, the probability density function $p(r)$ eventually reaches a steady state, which we ascertained by doubling the simulation time $t_{\text{max}}$ and verifying that $p(r)$ remained virtually unchanged. We make the empirical observation that an invariant measure seems to exist even for $\alpha < 3$, and that its form does not change qualitatively in the neighborhood of $\alpha = 3$. Our simulations thus suggest that condition~\eqref{cond:G:e^G>grad^2.G} is a technical condition that likely could be relaxed in future work.

For the values of $m$, $\sigma$, $\alpha$, and the forms of the potential $U(x)$ and pilot-wave force $H(x)$ shown here, the invariant measure $p(r)$ evidently has a relatively simple form with a single peak roughly located at $r\approx \sqrt{\alpha}$, the distance from the origin at which the attractive and repulsive forces on the walker are balanced, $r\approx \alpha/r$. In future work, we will seek to characterize the dependence of the form  of $p(r)$ on the strengths of the regular and singular potentials, as well as the walker mass and noise strength.

\begin{figure}[ht]
\begin{center}
\includegraphics[width=1\textwidth]{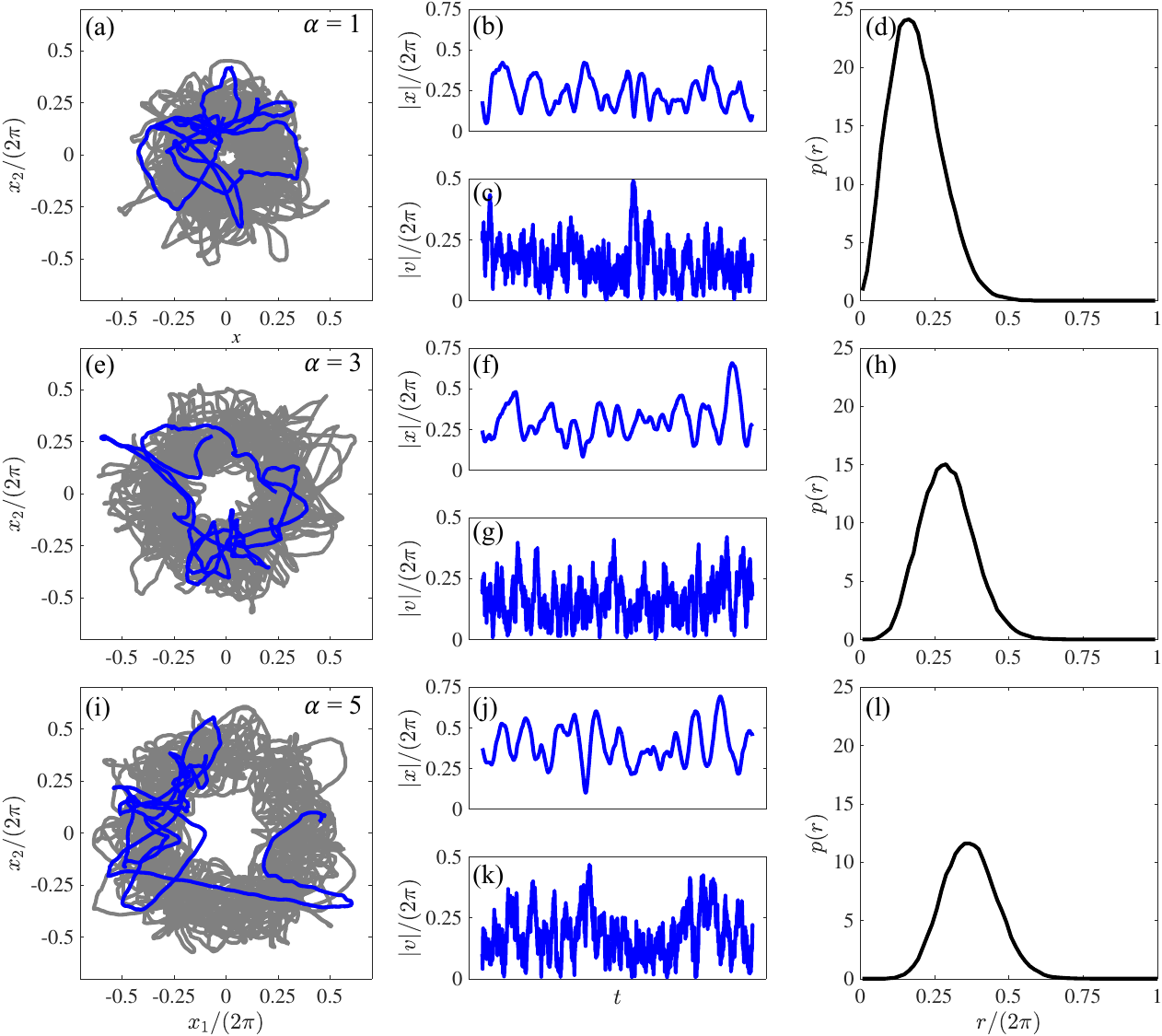}
\caption{Numerical simulations of Eq.~\eqref{eqn:droplet:original} for $m = 1$, $\sigma=1$, $U(x) = |x
|^2/2$ and $G(x) = -\alpha\log(|x|)$, for $\alpha=1$ (panels (a) through (d)), $\alpha=3$ (panels (e) through (h)) and $\alpha=5$ (panels (i) through (l)). In the leftmost panels, the gray curves show the two-dimensional walker position $x(t)=(x_1(t),x_2(t))$ over the time interval $t_{\text{max}} -1000 < t < t_{\text{max}}$. The blue curve highlights a subset of this trajectory over the interval $t_{\text{max}} - 50 < t < t_{\text{max}}$. The corresponding time series for the radius $r=|x(t)|$ and speed $|v(t)|$ are shown in the middle panels. The rightmost panels show the walker's radial position probability density function $p(r)$. The variables are divided by $2\pi$, which corresponds to the approximate periodicity of the pilot-wave force function $H(x)=\text{J}_1(x)x/|x|$.
}
 \label{Fig:Numerics}
 \end{center}
 \end{figure}
 
\section*{acknowledgment} HN acknowledges support from a start-up fund by the University of Tennessee Knoxville. AO acknowledges support from NSF DMS-2108839 and DMS-2510304. The authors would like to thank anonymous referees for their providing a thorough review of this work. We appreciate their careful reading and insightful comments, which have improved the manuscript.

\section*{conflict of interest}

The authors have no conflicts of interest to declare that are relevant to the content of this article.

\section*{Data availability}
Numerical simulation data is available upon request.

\appendix

\section{Auxiliary results}

In this section, we collect suitable estimates on the auxiliary system \eqref{eqn:Langevin:droplet} through Lemma \ref{lem:Langevin:E.exp(sup.Phi(x,v))} and Lemma \ref{lem:irreducible:control-xtilde}. In particular, the former establishes an exponential moment bound that is directly employed in the proof of Proposition \ref{prop:irreducible}, whereas the latter provides a control argument which appears in the proof of Lemma \ref{lem:irreducible:Langevin}.

\begin{lemma} \label{lem:Langevin:E.exp(sup.Phi(x,v))} Let $(x_1,v_1)$ be the solution of the Langevin system \eqref{eqn:Langevin:droplet:x}-\eqref{eqn:Langevin:droplet:v} with initial data $(x,v)\in \rbb^d\setminus\{0\}\times \rbb^d$. Then, for all $\beta>0$ sufficiently small
\begin{align} \label{ineq:Langevin:E.exp(sup.Phi(x,v))}
\E\exp\Big\{\beta \sup_{s\in[0,t]} \Phi(x_1(s),v_1(s))   \Big\} \le  2\exp \big\{\beta \Phi(x,v)+C\beta\,t\big\},\quad t\ge 0.
\end{align}
In the above, $\Phi$ is defined in \eqref{form:Phi} and $C$ is a positive constant independent of $\beta$ and $t$.

\end{lemma}
\begin{proof}
Denote by $\L_1$ the Kolmogorov operator associated with \eqref{eqn:Langevin:droplet:x}-\eqref{eqn:Langevin:droplet:v}, namely,
\begin{align*} %\label{form:L_1}
\L_1 g &:= \la v,\partial_x g\ra+\big\la -v-\grad U(x)- \grad G(x(t)),\partial_v g\big\ra +\frac{1}{2} \triangle_v g .
\end{align*}
Similarly to the proof of Lemma \ref{lem:moment-bound}, part 1., cf. \eqref{ineq:d.Phi(t)}, we apply $\L_1$ to $\Phi(x,v)$ and obtain the estimate
\begin{align*}
\L_1\Phi(x,v) \le -c\Big( |v|^2 +|x|^{q_0}+|x|^{-\beta_1}\Big) + C.
\end{align*}
In particular, for all $\beta>0$, It\^o's formula yields
\begin{align*}
\d \big(\beta\Phi(x_1(t),v_1(t))\big) & =  \beta\L_1\Phi(x_1(t),v_1(t))\d t+\d M_1(t)\\
& \le  -c\beta \Big( |v_1(t)|^2 +|x_1(t)|^{q_0}+|x_1(t)|^{-\beta_1}\Big)\d t+ C\beta\d t+ \d M_1(t),
\end{align*}
where the semi-Martingale process $M_1(t)$ is defined as
\begin{align} \label{form:M_1(t)}
M_1(t) =\beta\int_0^ t\Big\la v_1(r)+\kappa x_1(r)+\frac{x_1(r)}{|x_1(r)|} , \d W(r) \Big\ra,
\end{align}
whose quadratic variation process is given by
\begin{align} \label{form:M_1(t):quadratic-variation}
\la M_1\ra(t)=\beta^2\int_0^t \Big|v_1(r)+\kappa x_1(r)+\frac{x_1(r)}{|x_1(r)|}\Big|^2\d r.
\end{align}
Since $q_0\ge 2$, cf. condition \eqref{cond:q_0}, we note that
\begin{align*}
\d\la M_1\ra(t) \le c\beta^2 \big(|v_1(t)|^2 +|x_1(t)|^{q_0}+1\big)\d t.
\end{align*}
As a consequence, we obtain
\begin{align} \label{ineq:Langevin:d.Phi(x_1,v_1)}
\d\big( \beta\Phi(x_1(t),v_1(t))\big) \le  C\beta\d t+ \d M_1(t)-\frac{c}{\beta}\d \la M_1\ra(t),
\end{align}
for some positive constants $C,c$ independent of $\beta$. Next, we recall the exponential Martingale inequality:
\begin{align*}
\P\left( \sup_{t\ge 0}\left[M_1(t)-\frac{1}{2}\cdot \lambda\cdot \la M_1\ra(t)\right]\ge R \right) \le e^{-\lambda R},\quad \lambda>0,\, R>0.
\end{align*}
Based on the right-hand side of \eqref{ineq:Langevin:d.Phi(x_1,v_1)}, we choose $\lambda=c/\beta$ and observe that
\begin{align*}
\P\left( \exp\left\{\sup_{t\ge 0}\left[M_1(t)-\frac{c}{\beta} \la M_1\ra(t)\right]\right\}\ge R \right) \le R^{-2c/\beta},\quad  R>0.
\end{align*}
By taking $\beta$ sufficiently small, the above estimate implies
\begin{align*}
&\E\exp\left\{\sup_{t\ge 0}\left[M_1(t)-\frac{c}{\beta} \la M_1\ra(t)\right]\right\}\\
& =\int_0^\infty \P\left( \exp\left\{\sup_{t\ge 0}\left[M_1(t)-\frac{c}{\beta} \la M_1\ra(t)\right]\right\}\ge R \right)\d R \le 2.
\end{align*}
It follows from \eqref{ineq:Langevin:d.Phi(x_1,v_1)} that
\begin{align*}
\E \exp \Big\{\sup_{s\in[0,t]}\beta \Phi(x_1(s),v_1(s))    \Big\} \le 2 \exp \big\{\beta \Phi(x,v)+C\beta \,t    \big\}.
\end{align*}
This produces \eqref{ineq:Langevin:E.exp(sup.Phi(x,v))}, as claimed.

\end{proof}

\begin{lemma} \label{lem:irreducible:control-xtilde}  %Let $(\xt_1,\vt_1,\Gamma)$ be the solution of and $e_1=(1,0,\dots,0)$ be a unit element in $\rbb^d$. 

Given $R>2, t>1$, $r\in(0,1)$ and $(x_0,v_0)\in\rbb^d\setminus\{0\}\times \rbb^d$ such that
\begin{align*}
|x_0|+|x_0|^{-1}+|v_0|\le R,
\end{align*}
there exists a triple $(\xt_1,\vt_1,\Gamma)$ satisfying the deterministic control problem \eqref{eqn:Langevin:droplet:xtile} with
\begin{align} \label{cond:xtile_1:boundary}
(\xt_1(0),\vt_1(0)) = (x_0,v_0),\quad (\xt_1(t),\vt_1(t))=(e_1,0),
\end{align}
and 
\begin{align} \label{cond:int_0^t|xtile_1|<r}
\int_0^t |\xt_1(s)|^{p_2}\emph{d} s < r,
\end{align}
where $e_1=(1,0,\dots,0)$ is a unit element in $\rbb^d$.

\end{lemma}
\begin{proof} There are two cases to be considered depending on the relationship between $x_0$ and $e_1$: %position of $x_0$ in term of $e_1$.

{\bf Case 1}: $x_0$ is not a multiple of $e_1$, i.e.,
\begin{align*}
x_0\notin \{z\in\rbb^d:z=ce_1\text{ for some }c\in\rbb\}.
\end{align*}
Letting $\varepsilon$ small be given and chosen later, we first consider the time interval $[0,\varepsilon]$ and construct a pair $(\xt_1,\vt_1)$ on this interval satisfying
\begin{align*} %\label{cond:xtile_1'=vtile_1}
\dds\xt_1(s)= \vt_1(s),\quad 0\le s\le \varepsilon,
\end{align*}
with boundary conditions
\begin{align*} 
\xt_1(0)&=x_0,\quad \xt_1(\varepsilon)=e_1,\nt \\
\vt_1(0)&=v_0,\quad \vt_1(\varepsilon)=0. 
\end{align*}
To do so, letting $y\in \rbb^d\setminus\{0\}$ and $a\in(0,\varepsilon)$ be given and chosen later, we consider a path $\vt_1\in C([0,\varepsilon];\rbb^d)$ defined as
\begin{align*}
\vt_1(s) = y\frac{s(s-\varepsilon)}{a(a-\varepsilon)}+v_0\frac{(s-a)(s-\varepsilon)}{a\varepsilon},\quad 0\le s\le \varepsilon.
\end{align*}
Note that with the above choice, we readily have
\begin{align*}
\vt_1(0)=v_0,\quad \vt_1(\varepsilon)=0.
\end{align*}
Setting
\begin{align*}
\xt_1(s)&=x_0+\int_0^s \vt_1(s)\d s,
\end{align*}
a routine calculation yields
\begin{align*}
\xt_1(s)&=x_0+\int_0^s y\frac{s(s-\varepsilon)}{a(a-\varepsilon)}+v_0\frac{(s-a)(s-\varepsilon)}{a\varepsilon}\d s\\
&= x_0 + \frac{y}{a(a-\varepsilon)}\Big(\frac{s^3}{3}-\frac{\varepsilon s^2}{2}\Big) +\frac{v_0}{a\varepsilon}\Big(\frac{s^3}{3}-(a+\varepsilon)\frac{s^2}{s} +a\varepsilon s  \Big).
\end{align*}
Picking 
\begin{align*}
a=\frac{\varepsilon}{2},\quad \text{and}\quad y = \Big(e_1-x_0-\frac{\varepsilon}{6}v_0\Big)\frac{3}{2\varepsilon},
\end{align*}
we note that $\xt_1(\varepsilon)=e_1$. Furthermore, we obtain the following expressions for $s\in[0,\varepsilon]$:
\begin{align} \label{form:xtilde_1(s):case_1:[0,epsilon]}
\xt_1(s)&=x_0 - \Big(e_1-x_0-\frac{\varepsilon}{6}v_0\Big)\Big(2\frac{s^3}{\varepsilon^3}-3\frac{ s^2}{\varepsilon^2}\Big) +\frac{2v_0}{\varepsilon^2}\Big(\frac{s^3}{3}-\frac{3}{4}\varepsilon s^2  \Big)+v_0s, \nt \\
\vt_1(s)&= - \Big(e_1-x_0-\frac{\varepsilon}{6}v_0\Big)\frac{6}{\varepsilon}\Big(\frac{s^2}{\varepsilon^2}-\frac{ s}{\varepsilon} \Big)+2v_0\Big( \frac{s^2}{\varepsilon^2}-\frac{3}{2}\frac{s}{\varepsilon}\Big)+v_0. 
\end{align}
It is important to point out that the choice of $\xt_1$ in \eqref{form:xtilde_1(s):case_1:[0,epsilon]} satisfies that $\xt_1(s)\neq 0$ for $s\in[0,\varepsilon]$. To see this, we recast $\xt_1$ as follows:
\begin{align} \label{form:xtilde_1(s):case_1:[0,epsilon]:recast}
\xt_1(s)& = x_0\Big[1-\Big(3\frac{s^2}{\varepsilon^2}-2\frac{s^3}{\varepsilon^3}\Big)\Big]+e_1\Big[3\frac{s^2}{\varepsilon^2}-2\frac{s^3}{\varepsilon^3}\Big] +\varepsilon v_0\Big( \frac{s^3}{\varepsilon^3}-2\frac{s^2}{\varepsilon^2}+\frac{s}{\varepsilon}\Big).
\end{align}
On the one hand, since $s\in[0,\varepsilon]$ and $|v_0|<R$, it holds that
\begin{align*}
\Big|\varepsilon v_0\Big( \frac{s^3}{\varepsilon^3}-2\frac{s^2}{\varepsilon^2}+\frac{s}{\varepsilon}\Big) \Big|\le 4\varepsilon R.
\end{align*}
On the other hand, since in this case $x_0$ is not a multiple of $e_1$, we have
\begin{align*}
&\min_{s\in[0,\varepsilon]}\Big|x_0\Big[1-\Big(3\frac{s^2}{\varepsilon^2}-2\frac{s^3}{\varepsilon^3}\Big)\Big]+e_1\Big[3\frac{s^2}{\varepsilon^2}-2\frac{s^3}{\varepsilon^3}\Big]\Big|\\
&=  \min_{r\in[0,1]}|x_0(1-r)+e_1r|= c>0,
\end{align*}
for some positive constant $c=c(x_0,e_1)$ independent of $\varepsilon$. Taking $\varepsilon$ sufficiently small produces the bound
\begin{align*}
|\xt_1(s)|\ge c-4\varepsilon R>0,
\end{align*}
implying that $\xt_1(s)\neq 0$ for $0\le s\le \varepsilon$.

Next, let $\psi:\rbb\to\rbb$ be an arbitrarily smooth cut-off function satisfying
\begin{align} \label{form:psi_1}
\psi_1(t)=\begin{cases} \varepsilon,& t\le 0,\\
\text{monotone},& 0\le t\le 1,\\
1,& t\ge 1. \end{cases}
\end{align}
On the interval $[\varepsilon,t]$, we set
\begin{align} \label{form:xtilde_1(s):case_1:[epsilon,t]}
\xt_1(s) &=\begin{cases} \Big(1+\varepsilon-\psi_1\big(\frac{s}{\varepsilon}-1\big)\Big)e_1, & \varepsilon\le s\le 2\varepsilon,\\
\varepsilon e_1, &2\varepsilon\le s\le t-\varepsilon,\\
\psi_1\big(\frac{s-t+\varepsilon}{\varepsilon}\big)e_1,& t-\varepsilon\le s\le t.
\end{cases} \\
\vt_1(s)& =\xt_1'(s).\nt
\end{align}
From \eqref{form:xtilde_1(s):case_1:[0,epsilon]}--\eqref{form:xtilde_1(s):case_1:[epsilon,t]}, we see that $\xt_1\in C^1([0,t];\rbb)$ and $\vt_1\in C([0,t];\rbb)$ satisfy \eqref{cond:xtile_1:boundary}. 

Next, with regard to the condition \eqref{cond:int_0^t|xtile_1|<r}, we employ expression \eqref{form:xtilde_1(s):case_1:[0,epsilon]:recast} to see that
\begin{align*}
\int_0^\varepsilon |\xt_1(s)|^{p_2}\d s \le C \varepsilon R^{p_2},
\end{align*} 
for some positive constant $C$ independent of $\varepsilon$ and $R$. Also, from \eqref{form:xtilde_1(s):case_1:[epsilon,t]}, we have
\begin{align*}
\int_\varepsilon^t |\xt_1(s)|^{p_2}\d s &=\Big\{\int_\varepsilon^{2\varepsilon}+\int_{2\varepsilon}^{t-\varepsilon}+\int_{t-\varepsilon}^t\Big\}|\xt_1(s)|^{p_2}\d s\le  \varepsilon  +t\varepsilon^{p_2}\varepsilon\le 3\varepsilon t .
\end{align*}
It follows that
\begin{align*}
\int_0^t |\xt_1(s)|^{p_2}\d s \le C\varepsilon t R^{p_2}.
\end{align*}
By shrinking $\varepsilon$ further to zero if necessary, the above estimate produces \eqref{cond:int_0^t|xtile_1|<r}, as claimed.

Turning back to the control problem \eqref{eqn:Langevin:droplet:xtile}, it remains to pick a suitable control process $\Gamma$ so as to ensure that the triple $(\xt_1,\vt_1,\Gamma)$ satisfies \eqref{eqn:Langevin:droplet:xtile}. To this end, let $\Gamma$ be defined as
\begin{align} \label{form:Gamma}
\Gamma(s) = \vt_1(s)-v_0+\int_0^s \vt_1(\ell)+\grad U(\xt_1(\ell))+\grad G(\xt_1(\ell))\d \ell.
\end{align}
It is clear that $(\xt_1,\vt_1,\Gamma)$ solves \eqref{eqn:Langevin:droplet:xtile} on the time interval $[0,t]$, thereby finishing the proof in Case 1.

{\bf Case 2}: $x_0$ is a multiple of $e_1$. There are two sub cases depending on the dimension $d$.

\hspace{0.4cm} Case 2a: dimension $d=1$. Since in this case the initial condition satisfies $x_0>0$, we may employ the same argument as in Case 1 to obtain the triple $(\xt_1,\vt_1,\Gamma)$ satisfying \eqref{eqn:Langevin:droplet:xtile} as well as the conditions \eqref{cond:xtile_1:boundary} and \eqref{cond:int_0^t|xtile_1|<r}.

\hspace{0.4cm} Case 2b: dimension $d\ge 2$. In this case, denoting $e_2=(0,1,0,\dots,0)$, we observe that $x_0$ is not a multiple of $e_2$. Following the same procedure for \eqref{form:xtilde_1(s):case_1:[0,epsilon]}, we first drive $(x_0,v_0)$ at time 0 to $(e_2,0)$ at time $\varepsilon$. That is, we obtain the following path: 
 \begin{align} \label{form:xtilde_1(s):case_2b:[0,epsilon]}
\xt_1(s)&=x_0 - \Big(e_2-x_0-\frac{\varepsilon}{6}v_0\Big)\Big(2\frac{s^3}{\varepsilon^3}-3\frac{ s^2}{\varepsilon^2}\Big) +\frac{2v_0}{\varepsilon^2}\Big(\frac{s^3}{3}-\frac{3}{4}\varepsilon s^2  \Big)+v_0s, \nt \\
\vt_1(s)&= - \Big(e_2-x_0-\frac{\varepsilon}{6}v_0\Big)\frac{6}{\varepsilon}\Big(\frac{s^2}{\varepsilon^2}-\frac{ s}{\varepsilon} \Big)+2v_0\Big( \frac{s^2}{\varepsilon^2}-\frac{3}{2}\frac{s}{\varepsilon}\Big)+v_0,\quad 0\le s\le \varepsilon,
\end{align}
satisfying the boundary condition
\begin{align*} 
\xt_1(0)&=x_0,\quad \xt_1(\varepsilon)=e_2,\nt \\
\vt_1(0)&=v_0,\quad \vt_1(\varepsilon)=0. 
\end{align*}
On the interval $[\varepsilon,2\varepsilon]$, denoting
\begin{align*}
\psi_2(s)=  \frac{1}{2}\Big[ \cos \Big(\pi\frac{s}{\varepsilon}\Big)+1 \Big], \quad \varepsilon\le  s\le 2\varepsilon,
\end{align*}
 we choose
\begin{align} \label{form:xtilde_1(s):case_2b:[epsilon,2epsilon]}
\xt_1(s)&=\big( \psi_2(s) ,1-\psi_2(s),0\dots,0  \big),\\
\vt_1(s)&=\xt_1'(s),\quad \varepsilon\le s\le 2\varepsilon.\nt 
\end{align}
Observe that with this choice, 
\begin{align*}
\xt_1(2\varepsilon)=e_1,\quad \vt_1(2\varepsilon)=0.
\end{align*}
In other words, $(\xt_1,\vt_1)$ drives $(e_2,0)$ at time $\varepsilon$ to $(e_1,0)$ at time $2\varepsilon$. Next, concerning the interval $[2\varepsilon,t]$, similar to \eqref{form:xtilde_1(s):case_1:[epsilon,t]}, we obtain the following expression:
\begin{align} \label{form:xtilde_1(s):case_2b:[2epsilon,t]}
\xt_1(s) &=\begin{cases} \Big(1+\varepsilon-\psi_1\big(\frac{s}{\varepsilon}-2\big)\Big)e_1, & 2\varepsilon\le s\le 3\varepsilon,\\
\varepsilon e_1, &3\varepsilon\le s\le t-\varepsilon,\\
\psi_1\big(\frac{s-t+\varepsilon}{\varepsilon}\big)e_1,& t-\varepsilon\le s\le t,
\end{cases} \\
\vt_1(s)& =\xt_1'(s),\nt
\end{align}
where $\psi_1$ is as in \eqref{form:psi_1}. 

Now, from \eqref{form:xtilde_1(s):case_2b:[0,epsilon]}, \eqref{form:xtilde_1(s):case_2b:[epsilon,2epsilon]} and \eqref{form:xtilde_1(s):case_2b:[2epsilon,t]}, it is not difficult to verify that $(\xt_1,\vt_1)$ satisfies the estimate \eqref{cond:int_0^t|xtile_1|<r} as well as the boundary condition \eqref{cond:xtile_1:boundary}. Lastly, by choosing $\Gamma$ as in \eqref{form:Gamma}, we obtain the triple $(\xt_1,\vt_1,\Gamma)$ solving \eqref{eqn:Langevin:droplet:xtile}, thereby finishing the construction in Case 2. The proof is thus complete.
\end{proof}

\bibliographystyle{abbrv}
\bibliography{singular_bib,ARFMBib}

\end{document}